\documentclass[12pt]{article}
 
\usepackage{amsfonts,latexsym,amsmath,amssymb,amsthm}
\usepackage[mathscr]{eucal}
\usepackage{graphicx,color}
\usepackage{hyperref}
\usepackage{epstopdf}
\usepackage[margin = 1in]{geometry}
\usepackage{multirow}
\newtheorem{theorem}{Theorem}
\newtheorem{lemma}{Lemma}
\newtheorem{proposition}{Proposition}
\newtheorem{corollary}{Corollary}

\newtheorem{remark}{Remark}

\newtheorem{example}{Example}

\def\N{{\mathbb N}}    
\def\R{{\mathbb R}}    
\def\C{{\mathbb C}}  
\def\sti{0.15in}
\def\y{\mathbf{y}}
\def\M{\mathbf{M}}
\def\E{\mathbf{E}}
\def\bS{\mathbb{S}}
\def\U{\mathbf{U}}

\def\u{\mathbf{u}}
\def\z{\mathbf{z}}

\def\v{\mathbf{v}}

\def\V{\mathbf{V}}

\def\saK{\mathbf{K}}
\def\A{\mathbf{A}}
\def\B{\mathbf{B}}

\def\x{\mathbf{x}}
\def\t{\mathbf{t}}

\def\d{\mathrm{d}}
   
\newcommand{\Vector}[1]{\ensuremath{\boldsymbol{#1}}}
\newcommand{\gammab}{\Vector{\gamma}}
\newcommand{\phib}{\Vector{\phi}}
\newcommand{\todo}[1]{\begin{color}{red}#1\end{color}}
\newcommand{\victor}[1]{\textcolor{black}{#1}}
\newcommand{\jean}[1]{\textcolor{black}{#1}}
\newcommand{\oz}[1]{\textcolor{black}{#1}}
\usepackage{soul}
\begin{document}

\title{ \LARGE \bf%
\jean{Minimizing rational functions: A hierarchy of approximations  via pushforward measures}}

\author{Jean Bernard Lasserre$^{1}$, Victor Magron$^{1}$, Swann Marx$^{2}$ {and Olivier Zahm$^{3}$}}

\footnotetext[1]{CNRS; LAAS; 7 avenue du colonel Roche, F-31400 Toulouse; France}
\footnotetext[2]{CNRS; LS2N $\&$ \'Ecole Centrale de Nantes; 1 rue de la No\"e F-44321, Nantes; France}
\footnotetext[3]{Univ. Grenoble Alpes, Inria, CNRS, Grenoble INP, LJK, 38000 Grenoble, France}
\maketitle
\abstract{This paper is concerned with minimizing a sum of rational functions over a compact set of \oz{high-dimension}. Our approach relies on the second Lasserre's hierarchy (also known as the \emph{upper bounds} hierarchy) \oz{formulated on the \emph{pushforward} measure in order to work in a space of smaller dimension. We show that in the general case the minimum can be approximated as closely as desired from above with a hierarchy of semidefinite programs (SDPs) problems or, in the particular case of a single fraction, with a hierarchy of generalized eigenvalue problems.}
\jean{We numerically illustrate the 
potential of using the pushforward measure rather than
the standard upper bounds hierarchy. In our opinion, this potential should be  a 
strong incentive to investigate a related 
challenging problem interesting in its own; namely integrating an arbitrary power of a given polynomial on a simple set (e.g., unit box or unit sphere)
with respect to Lebesgue or Haar measure.}


\section{Introduction}


\paragraph{Problem statement.} We consider the following optimization problem
\begin{equation}
\label{quotients}
\rho = \min_{\mathbf x\in \mathbf K} \sum_{i=1}^N \frac{f_i(\mathbf{x})}{g_i(\mathbf{x})} ,
\end{equation}
where $\saK\subset\R^n$ is a compact set,  all numerators and denominators are polynomials, and all denominators are positive on $\saK$. 
{An important and motivating application} is the minimization of a \emph{generalized Rayleigh quotient} where $f_i(\x) = \x^\top A_i \x$ and $g_i(\x) = \x^\top B_i \x$ are quadratic functions defined from symmetric $n \times n$ matrices  $A_i,B_i$ (each $B_i$ having only positive eigenvalues) and with $\mathbf K=[-1,1]^n$ being the unit cube. Generalized Rayleigh quotients appear in many problems, such as  the multi-user MIMO system \cite{Primolevo2006}, sparse Fisher discriminant analysis in pattern recognition \cite{Wu2009} and in nonlinear dimension reduction methods for reduced order modeling as in \cite{Zahm2020}. 

\jean{Approximating as closely as desired the global minimum $\rho$ in Problem \eqref{quotients} is challenging because it is nonlinear,  nonconvex and, in addition, the dimension $n$ of $\x$  may be high. We discuss a methodology 
which combines i) a ``pushforward" technique \cite{Las19} to obtain a equivalent problem of reduced dimension and ii) the SOS-hierarchy of upper bounds introduced in \cite{lasserre2011new}. While 
this reduction to a hierarchy of small dimensional generalized eigenvalue problems reveals some potential, 
it also raises an interesting scientific challenge in its own: how to compute efficiently moments of the form
\begin{equation}
    \label{eq:integral}
    \int_{\saK} f^d\,d\lambda\,,\quad d\in\N\,,
\end{equation}
where 
$f$ is a polynomial, $\saK$ is a simple set (e.g., a box, ellipsoid, hypercube, or their image by an affine mapping) and $\lambda$ is any measure whose support is $\saK$ and  whose moments can be explicitly computed using cubature formula for integration. 
In a typical example, $\saK$ is the unit box $[-1,1]^n$ and $\lambda$ is the Lebesgue measure. One goal of this paper is to bring attention of the optimization community to problem \eqref{eq:integral} and convince that an efficient algorithm for \eqref{eq:integral} (even in some restricted setting for $f$ and $\saK$) would be very interesting for global optimization.}



\paragraph{Lower bounds hierarchy.} \jean{ Initially, \emph{the moment-sum-of-squares}
(Moment-SOS) hierarchy ~\cite{Lasserre01moments}, also known as the \emph{lower bounds} hierarchy, was designed to handle polynomial optimization problems, i.e. problems whose objective function and constraints are described with polynomials. Each step of the hierarchy
is a semidefinite relaxation of the original problem which can be solved efficiently\footnote{An SDP is a linear conic program on the convex cone of real symmetric matrices with nonnegative eigenvalues. With prescribed accuracy it can be solved via interior-point methods in time polynomial in the input size. For more details about applications of SDP together with complexity estimates of the interior-point algorithms, we refer to~\cite{deKlerkSDP,NN-94,Vandenberghe94SDP}}. 
Its dual has a simple interpretation in terms of SOS-based positivity certificates.
By increasing the degree of the SOS in the certificate (and therefore the size of the resulting SDP-relaxation) one obtains a monotone sequence of lower bounds which converges to the global minimum.}

\jean{The Moment-SOS hierarchy can be applied to solve the
\emph{Generalized Moment Problem} (GMP) with semi-algebraic data\footnote{A GMP is a linear conic optimisation problem on convex cones of finite Borel measures.
In a GMP with algebraic data, all functions and supports of measures that appear in its description are semi-algebraic}. In fact, Polynomial  optimization is only one of the numerous applications of the GMP and for more details
the interested reader is referred to \cite{lasserre2010moments,henrion2020moment}.
For problem \eqref{quotients} with a single fraction, the SOS-hierarchy was first proposed in \cite{jibetean2006global}.
Of course, via a common denominator, the case of several fractions reduces to a single fraction. However the resulting high degree of some of the involved polynomials is a serious  obstacle to implement the standard SOS-hierarchy as proposed in \cite{jibetean2006global} and a specific approach, e.g. as in \cite{bugarin2016minimizing}, is needed.}
\jean{But in view of its high computational burden, the standard lower bounds hierarchy is restricted to problems of modest dimension.
Indeed, for problems }involving polynomials with $n$ variables of maximal degree $d$, the size of the resulting SDP relaxations grows rapidly as it is proportional to $\binom{n+d}{n}$. To overcome these limitations one may take advantage of {some structure of large-scale problems like symmetries and/or sparsity}. For instance, sparsity has been considered in~\cite{Las06SparseSOS,Waki06SparseSOS,chordaltssos,tssos}. To the best of our knowledge, such structures {have not been considered to solve \eqref{quotients} with the exception of \cite{bugarin2016minimizing}.}
{We here follow an approach different from sparsity exploiting schemes: we use a ``pushforward measure" technique to work in a space of smaller dimension.
This technique initiated in \cite{Las19} relies on
a second Lasserre's hierarchy which now provides a monotone sequence of \emph{upper bounds} which converges to  the minimum.}.


\paragraph{Upper bounds hierarchy.}
\jean{A second SOS-based hierarchy proposed in \cite{lasserre2011new} yields a monotone sequence of \emph{upper bounds} which converges to the minimum and therefore can be seen as complementary to the first SOS hierarchy of lower bounds.
In addition, and in contrast to the hierarchy of lower bounds, the function to be minimized may not  be either a polynomial or a semialgebraic function.}
At each step of the hierarchy, an upper bound on the minimum of a given polynomial is computed by solving a so-called {\em generalized eigenvalue} problem. In our context, the \jean{two involved} matrices encode certain information regarding the moments of {some reference probability measure $\mu$ whose support is exactly \jean{the set of feasible solutions} $\saK$.}
For instance, in the bivariate case, the entries of these matrices at the second step of the hierarchy necessarily depend on the value of the integrals $\int_\saK x_1 d \mu$, $\int_\saK x_2 d \mu$, $\int_\saK x_1^2 d \mu$, $\int_\saK x_1 x_2 d \mu$ and $\int_\saK x_2^2 d \mu$.
In several {important cases these values are available analytically.}
This includes the case where $\mu$ is the uniform
measure on, for instance, a hypercube $\saK=[-1, 1]^n$, a simplex, {or their image by an affine mapping. For more details about these closed formula, the interested reader is referred to \cite{Grundmann78,Parsimony16,deKlerk16} and also to
\cite{magron2018interval} and \cite{de2019distributionally} for applications in computer arithmetic and robust optimization, respectively.}

Several efforts have been made to provide convergence rates  for the  hierarchy of upper bounds.
In~\cite{deKlerk16}, the authors obtain convergence rates which are no worse than $O(1/\sqrt{d})$ and often match practical experiments. On some specific sets $\mathbf K$ this convergence rate has been improved. \jean{For instance, 
for the box $\saK=~[-1,1]^n$ and the sphere $\saK=\mathbb{S}^{n-1}$, an $O(1/d^2)$ rate of convergence  has been obtained in \cite{de2017improved} and \cite{de2020convergence} respectively.
For some other cases (convex bodies in particular)
an $O(\log^2d/d^2)$ rate of convergence rates has been recently obtained in \cite{slot2020improved} and in \cite{laurent2020near}.}
All these research efforts show that the asymptotic behavior of the upper bounds hierarchy is \jean{better understood than for} the lower bounds hierarchy. 

\jean{As for the lower bounds hierarchy, the size of the resulting matrices is
critical and restricts its application to small size problems. In fact so far its main interest has been its theoretical rate of convergence to the global minimum as such guarantees are rather rare.
A first attempt to break the curse of dimensionality in the upper bounds hierarchy for optimization has been done by the first author  in \cite{Las19}. The idea is to use the \emph{pushforward} measure of the Lebesgue measure by
the polynomial to minimize. In doing so one reduces the initial problem to a related \emph{univariate} problem and as a result one obtains a hierarchy
of upper bounds (again generalized eigenvalue problems) which involves univariate SOS polynomials {of increasing degree}. Again and remarkably, in \cite{laurent2020near} the authors
have shown a $O(\log^2d/d^2)$ rate of convergence to the global minimum, which
makes this ``univariate" hierarchy appealing as the computational burden of the resulting eigenvalue problems are orders of magnitude smaller than the 
initial (multivariate) hierarchy of upper bounds. Indeed at step $d$
one handles eigenvalue problems with matrices of size $d+1$ 
instead of ${n+d\choose d}$.}

\oz{
The price to pay for this highly desirable dimensionality break is that one needs to compute integrals of the form \eqref{eq:integral} in order to assemble the underlying moment matrices.
In a sense, the computational burden has switched from solving a large eigenvalue problem to computing high-dimensional integrals.
When $\saK$ is sufficiently simple (\emph{e.g.} a hypercube, a sphere, or their image by an affine map) these integrals can be carried out analytically by expanding the terms in the monomial basis.
Such an expansion is, however, tedious and very costly even for moderate size $n$ and reasonable power $d$. 
}


We also mention several applications of pushforward measures in the context of lower bounds hierarchies, allowing one to approximate the volume of a semialgebraic set \cite{lasserre2019volume}, the image of a semialgebraic set by a polynomials map \cite{magron2015semidefinite} as well as 
its generalization to reachable sets of discrete-time polynomial systems  \cite{magron2019semidefinite}.
\paragraph{Contribution.}
\oz{We propose an \emph{upper bounds} hierarchy for solving \eqref{quotients} where, similarily to \cite{Las19}, we break the dimension by using pushforward measures.
Contrarily to \cite{Las19}, the dimension now reduces from $n$ to $2N$, where $N$ is the number of fractions in \eqref{quotients}. This approach is thus relevant when $N\ll n$.
Our method requires the computation of integrals of the form
\begin{equation}\label{eq:integral2}
\oz{\int_{\mathbf{K}} \prod_{i=1}^N f_i(\mathbf{x})^{\alpha_i} g_i(\mathbf{x})^{\beta_i} \d\lambda (\mathbf{x}),\quad\alpha,\beta\in\N^N\,,}
\end{equation}
instead of \eqref{eq:integral}.
Here $\lambda$ is any measure whose support is $\saK$ and with all moments available or easy to compute, for instance the Lebesgue measure on $\saK=[-1,1]^n$. Again, computing such integrals is not an easy task, except {for specific sets $\mathbf{K}$}.
Indeed, to the best of our knowledge, there is no efficient method to compute such information for general semialgebraic sets. But when $\mathbf{K}$ is ``simple'' enough (e.g., the unit sphere $\mathbb S^{n-1}$ or the box $[-1,1]^n$), such integrals can be computed exactly and are even trivial for small dimension and small degree.
}

\jean{We again emphasize that computing integrals \eqref{eq:integral2} is challenging 
even for problem \eqref{quotients} with a single fraction $N=1$. Indeed for instance a brute force expansion of the integrand in the monomial basis 
is rapidly tedious and out of reach. This issue is not addressed 
in the present paper but we provide some research directions for 
further investigation.}

Several results are provided in this paper: In particular we prove that the minimum of a single rational function $N=1$ can be approximated from above and as closely as desired, by solving a hierarchy of semidefinite programs involving bivariate SOS polynomials.
When $N>1$  we do not obtain a sequence of certified upper bounds any more,
but the resulting hierarchy of semidefinite programs still converges to the minimum. In preliminary numerical experiments
one obtains better approximations than with the classical upper bounds hierarchy, and in a significantly more efficient way when $N=1$.
\jean{We do not claim that this method is competitive with efficient local optimization frameworks. Indeed so far
it is limited to problems of modest size, i.e., with a small number of quotients, due to the difficulty of computing integrals 
\eqref{eq:integral2} with large degree and/or number of variables. We rather suggest that
that one can approximate the minimum in \eqref{quotients} more efficiently with the pushforward measure
than with the standard (multivariate) upper bounds hierarchy. In view of the nice convergence properties of the (univariate) upper-bound hierarchy proved in \cite{laurent2020near}, efficient methods for 
computing  \eqref{eq:integral2} might have an impact for solving \eqref{quotients} efficiently.  Therefore we hope to convince the reader that further investigation of \eqref{eq:integral2} even for special classes of polynomials and sets $\saK$ should  deserve more attention. For instance, the work in \cite{baldoni} is already useful when $\saK$ is a simplex.}


\paragraph{Organization.} The paper is organized as follows. Section \ref{sec:notation} introduces some notation together with useful technical results. The case of a single rational function is treated in Section \ref{sec:singlerat}. Section \ref{sec:multirat} is devoted to the case of a sum of rational functions. Our contributions are illustrated by some numerical experiments in Section \ref{sec:benchs}. Finally, Section \ref{sec:conclusion} collects some concluding remarks and introduces further research lines to be followed.
\section{Notation and useful results}
\label{sec:notation}
\paragraph{Borel measures, moment and localizing matrices.}
{Following \cite[Chapter 3]{lasserre2010moments}}
we consider polynomial and rational functions of the variable $\x=(x_1,\dots,x_n)$.
Let $\R[\x]$ be the vector space of polynomials with real coefficients, and given $d\in\N$, we note $\R[\x]_{d}$ its restriction to polynomials of \oz{total} degree at most $d$. 
\oz{We denote by $\Sigma[\x]=\left\{\sum_{i=1}^m h_i(\x)^2 : m\in\mathbb{N}, h_i\in\R[\x]\right\}$ the set of sum of squares (SOS) polynomials, and by $\Sigma[\x]_d$ its restriction to polynomials of degree at most $2 d$.}
Given a compact set $\saK \subset \R^n$, let us denote by $\mathcal{M}(\saK)$ the set of Borel (i.e., positive) measures supported on $\saK$.

Given a subset $\A \subseteq \saK$ and two Borel measures $\mu,\nu \in \mathcal{M}(\saK)$, a measurable function $f:\A\rightarrow \mathbb{R}_+$ such that $\nu(\A)=\int_{\A} f\,d\mu$ is called a \emph{density} of $\nu$ with respect to $\mu$.
In the particular case where $f \in \Sigma[\x]$, we refer to $f$ as an \emph{SOS density}.
Well-known examples of Borel measures include the Dirac measure, the uniform (also called Lebesgue) measures etc.
Given $\x \in \saK$, the Dirac measure $\delta_{\x}$ concentrated on $\x$ is defined by $\delta_\x(\A) :=  \textbf{1}_{\A}(\x)$, for all $\A \subseteq \saK$, with $\textbf{1}_{\A}$ being the indicator function on $\A$.
We define the restriction of the Lebesgue measure on such an $\A$ by $\d \lambda_{\A}(\x) := \textbf{1}_{\A} \d \x$, i.e., $\lambda_{\A}$ has density $\textbf{1}_{\A}$.
Given a measure $\mu \in \mathcal{M}(\saK)$, let $\mathbf{y}=(y)_{\alpha\in\mathbb{N}^n}$ be a real sequence whose entries are the moments of $\mu$, called its \emph{moment sequence}, i.e., $y_{\alpha} = \int_{\saK} \x^\alpha \d \mu(\x)$, for all $\alpha \in \N^n$.

For a given sequence $\y \in\mathbb{R}^{\mathbb{N}^n}$ we introduce the Riesz linear functional 
\begin{equation}
\begin{split}
L_{\mathbf{y}}:\: \mathbb{R}[\mathbf{x}] &\rightarrow \mathbb{R}\\
f\: \left(= \sum_{\alpha\in\mathbb{N}^n} f_\alpha \mathbf{x}^\alpha\right)&\mapsto L_{\mathbf{y}}(f) = \sum_{\alpha\in\mathbb{N}^n} f_\alpha y_\alpha.
\end{split}
\end{equation}

{With $d\in\N$, the truncated \emph{moment matrix} $\M_d(\y)$ associated with $\y$
is the real symmetric matrix with rows and columns indexed in the canonical basis ($\mathbf{x}^\alpha$) and with entries:} $$\mathbf{M}_d(\mathbf{y})(\alpha,\beta):= L_{\mathbf{y}}(\mathbf{x}^{\alpha + \beta})=y_{\alpha + \beta}\,,\quad \alpha,\beta\in\mathbb{N}_d^n,$$
{where $\mathbb{N}^n_d:=\lbrace \alpha\in\mathbb{N}^n\mid \alpha_i\leq d,\: i=1,\ldots,n\rbrace$. This matrix is the  multivariate version of a Hankel matrix.} 
\oz{Indeed with $n=1$ and $d=2$, the moment matrix is exactly a Hankel matrix:}
\begin{equation}
\mathbf{M}_2(\mathbf{y})  = \begin{bmatrix}
y_0 & y_1 & y_2\\
y_1 & y_2 & y_3\\
y_2 & y_3 & y_4
\end{bmatrix}.
\end{equation}
In the univariate case, if $\M_d(\y)\succeq0$ for all $d$, then
$\y$ has a representing measure on $\R$, i.e.,
$y_\alpha = \int_{\mathbb{R}} x^\alpha d\mu(x)$, for all $\alpha\in\mathbb{N}$.}

{Next, with $f\in\R[\x]$ in the form:}
\begin{equation}
\mathbf{x}\mapsto f(\mathbf{x})=\sum_{\gamma\in\mathbb{N}^n} f_\gamma \mathbf{x}^{\oz{\gamma}}\,,
\end{equation}
{the  \emph{localizing matrix} associated with $\mathbf{y}$ and $f$  is
the real symmetric matrix $\M_d(f\,\y)$ with rows and columns indexed in the canonical basis ($\mathbf{x}^\alpha$), and with entries:}
$$\mathbf{M}_d(f \mathbf{y})\oz{(\alpha,\beta)} = L_{\mathbf{y}}(f(\mathbf{x})\,\x^{\alpha + \beta}) \,=\,  \sum_{\gamma\in\mathbb{N}^n} f_\gamma \,y_{\mathbf{\gamma} + \mathbf{\alpha} + \mathbf{\beta}}\,, \quad\mathbf{\alpha},\: \mathbf{\beta}\in\mathbb{N}^n_d.$$ For example, given $n=1$, $d=2$ and the polynomial $f(\mathbf{x}) = a-x^2$:
\begin{equation}
\mathbf{M}_{2}(f\,\mathbf{y}) = \begin{bmatrix}
ay_0 - y_2 & ay_1 - y_3 & ay_2 - y_4\\
ay_1 - y_3 & ay_2 - y_4 & ay_3 - y_5\\
ay_2 - y_4 & ay_3 - y_5 & ay_4 - y_6
\end{bmatrix}\,.
\end{equation}

{Let us recall} a useful preliminary result stated in \cite[Theorem 3.2]{lasserre2011new}.
\begin{theorem}
\label{th:loc}
Let $\saK \subseteq [-1,1]^n$ be compact and let $\mu$ be an arbitrary, fixed, finite Borel measure supported on $\saK$ and with vector of moments $(\y_{\alpha})$, $\alpha \in \N^n$. Let $h$ be a continuous function on $\R^n$. 
Then $h$ is nonnegative on $\saK$ if and only if $\M_d(h \, \y) \succeq 0$, for all $d \in \N$.
\end{theorem}
{Now, fix an arbitrary Borel measure $\mu$ whose support is $\saK$, and with vector of moments $(\y_{\alpha})_{\alpha \in \N^n}$.
Invoking Theorem \ref{th:loc}, in \cite{lasserre2011new} the first author provides a monotone sequence of upper bounds converging 
to the minimum of a polynomial $f$ over a compact set $\saK$, 
by solving the hierarchy of SDPs indexed by $d \in \N$:}
\begin{align}
\label{eq:pol_mom}
\begin{array}{rl}
a_{d} = \displaystyle\sup_{a \in \R}& \,a: \\
\mbox{s.t.}& \M_{d}(f\, \y) \, \succeq \, a \, \M_d( \y) \,,
\end{array}
\end{align}
whose dual is given by
\begin{align}
\label{eq:pol_sos}
\begin{array}{rl}
\displaystyle\inf_{\sigma \in \Sigma[\x]_d}&  \, \displaystyle\int_{\saK} f(\mathbf{x})\,\sigma(\mathbf{x}) \,\d \mu(\mathbf{x}): \\
\mbox{s.t.}&  \displaystyle\int_{\saK}  \sigma(\mathbf{x})\, \d \mu \, = \, 1 \,, \\
\end{array}
\end{align}
%
\begin{theorem}{(\cite{lasserre2011new})}
\label{th:cvg_pol}
Let $\saK \subseteq \R^n$ be a compact set, $\mu \in \mathcal{M}(\saK)$ with moment sequence $\y$ and $f \in \R[\x]$.
Consider the hierarchy of semidefinite programs \eqref{eq:pol_mom} indexed by $d\in\N$.
Then:
\begin{itemize}
\item[(i)] SDP \eqref{eq:pol_mom} has an optimal solution $a_d \geq f^\star$ for every $d\in\N$.
\item[(ii)] There is no duality gap between SDP \eqref{eq:pol_sos} and SDP \eqref{eq:pol_mom}, and SDP \eqref{eq:pol_sos} has an optimal solution $\sigma^\star \in \Sigma[\x]_{d}$ which satisfies $\int_{\saK} (f(\x) - a_d) \sigma^\star(\x) \d \mu(\x) = 0$.
\item[(iii)] The sequence $(a_d)_{d \in \N}$ is monotone nonincreasing and $a_d \downarrow f^\star$ as $d\to \infty$.
\end{itemize}
\end{theorem}
{In fact, solving SDP \eqref{eq:pol_mom} boils down to solving a generalized eigenvalue problem for the pair of matrices $\M_d(f\,\y)$ and $\M_d(\y)$. More recently, in \cite{Las19} the first author has shown that 
$f^\star$ can also be approximated from above by considering a hierarchy of generalized eigenvalue problems indexed by $d$, but now involving Hankel matrices of size $d+1$ instead of ${n+d\choose d}$.} 
The entries of these matrices are linear in the moments of the \emph{pushforward measure} of the Lebesgue measure with respect to $f$. 
%
\paragraph{Pushforward measure.} 
\oz{Let $\mathbf{U}:\: \saK \rightarrow \mathbf{\Omega} \subseteq \mathbb{R}^m$ be a Borel measurable function. The \emph{pushforward measure} $\mathbf{U}_\# \mu$ of the measure $\mu\in\mathcal{M}(\saK)$ through $\mathbf{U}$ is the measure supported on $\U(\saK)$ defined by}
\begin{equation}
\mathbf{U}_\# \mu(C)=\mu(\mathbf{U}^{-1}(C)),
\end{equation}
for any $C\in \mathcal{B}(\Omega)$, where $\mathcal{B}(\mathbf{\Omega})$ denotes the Borel set of the compact set $\mathbf{\Omega}$, and $\U^{-1}(C)$ is the preimage of $C$ by the mapping $\U$.\\

In particular when $\U = f$ and $\mu = \lambda$, let $f_\# \lambda$ be the pushforward measure of the restriction of the Lebesgue measure $\lambda$ on $\saK$ with respect to $f$. 
Let us denote by $\y^\# = (y_d^\#)_{d \in \N}$ the sequence of moments
$$
 y_{d}^{\#} := \int_{[0,+\infty)} u^d \, \d f_\# \lambda (u) = \int_{\saK} f(\mathbf{x})^{d} \d\lambda \,.
$$
{As in \cite{Las19}, consider the hierarchy of SDP programs, indexed by $d\in\N$:}
\begin{align}
\label{eq:pol_pfm_mom}
\begin{array}{rl}
a_d^\# = \displaystyle\sup_{a \in\R}& \,a \\
\mbox{s.t.}& \M_d(u \, \y^\#) \, \succeq \, a \, \M_d( \y^\#)  \,.
\end{array}
\end{align}
Since the support of $f_\# \lambda$ is contained in the interval $[f^\star, +\infty)$, the results from \cite[Theorem 3.3]{lasserre2011bounding} imply that $a_d$ is an optimal solution of SDP \eqref{eq:pol_pfm_mom} for all $d\in\N$ and $a_d^\# \downarrow f^\star$ as $d \to \infty$ (see also \cite[Theorem 2.3]{Las19}).

In the sequel, we extend this framework based on the pushforward measure to the case of rational functions.
Using the pushforward measure is particularly interesting in the case where $N \ll n$, because the dimension of the initial problem can be drastically reduced.
{In particular, if $N=1$ we reduce to a $2$-dimensional problem.}
%
\section{Minimizing {a single} rational function}
\label{sec:singlerat}

In this section we consider the case $N=1$ in the sum given in \eqref{quotients}. 
The goal is thus to compute
\begin{equation}
\label{eq:minfg}
\rho = \min_{\mathbf{x}\in \saK} \frac{f(\mathbf{x})}{g(\mathbf{x})},
\end{equation}
where $f$ and $g$ are polynomials, $g$ being positive on the compact set $\mathbf{K}\subset\R^n$. 
By compactness of $\saK$, and since the function $\mathbf{x}\mapsto \frac{f(\mathbf{x})}{g(\mathbf{x})}$ is continuous, the minimum is attained for some $\mathbf{x}^\star\in \saK$ so that $\rho = f(\mathbf{x}^\star)/g(\mathbf{x}^\star)$.
\subsection{An instance of the generalized moment problem}
\label{sec:gmp}
{The next result from \cite{jibetean2006global} provides
an alternative expression for $\rho$ (see also  \cite[Section 5.8]{lasserre2010moments}). }
For the sake of completeness, we also recall its basic proof.
\begin{proposition}(\cite{jibetean2006global})
\label{prop:minfg_measure}
 The solution $\rho$ to \eqref{eq:minfg} is the optimal value of the following infinite-dimensional linear problem (LP):
\begin{align}
\label{eq:minfg_measure}
\begin{array}{rl}
\rho =\displaystyle\inf_{\mu \in \mathcal{M}(\saK)}&  \, \displaystyle\int_{\saK} f(\x) \,\d \mu(\x) \\
\mbox{s.t.}&  \displaystyle\int_{\saK} g(\x) \, \d \mu(\x) \, = \, 1 \,.
\end{array}
\end{align}
\end{proposition}
\begin{proof}
 Let $ \mu^\star = \frac{1}{g(\mathbf{x}^\star)}\delta_{\mathbf{x}^\star}$ be the Dirac measure centered at $\mathbf{x}^\star$ weighted by $1/g(\mathbf{x}^\star)$. We have $\int_{\saK} g(\mathbf{x})\d\mu^\star(\mathbf{x}) = 1$ so that the infimum in \eqref{eq:minfg_measure} is upper bounded by $\int_{\saK} f \d\mu^\star = \frac{f(\mathbf{x}^\star)}{g(\mathbf{x}^\star)} = \rho$. 
 Now, since $g$ is positive and by definition of $\rho$, we have $\rho g(\x)\leq f(\mathbf{x})$ for any $\x\in \saK$. Integrating against any measure $\mu\in\mathcal{M}(\saK)$ such that $\int g(\mathbf{x}) \d\mu(\mathbf{x})=1$ yields $\rho \leq \int_{\saK} f(\mathbf{x}) \d\mu(\mathbf{x})$, thus $\rho$ is upper bounded by the infimum in \eqref{eq:minfg_measure}, which concludes the proof.
\end{proof}
 {For more details about semidefinite relaxations of the infinite-dimensional LP \eqref{eq:minfg_measure} and its dual LP,
the interested reader is referred to \cite{jibetean2006global}.}

As in the case of polynomial minimization, one can fix a reference measure $\mu \in \mathcal{M}(\saK)$ (e.g., $\mu = \lambda$) with moment sequence $\y$.
Then, one way to approximate \eqref{eq:minfg_measure} from above is to replace the set of measures $\mathcal{M}(\saK)$ by a subset of measures with SOS polynomial densities (with respect to $\mu$) of degree at most $2d$, $d\in\N$. 
\if{
Let $\lambda$ be the Lebesgue measure on $K$ and denote by $\Sigma_n$ the set of SOS polynomials on $\R^n$, meaning all the polynomials written as the sum of squares of polynomials.
Restricting $\mu$ in \eqref{eq:minfg_measure} to be of the form $\d\mu=\sigma\d\lambda$ with $\sigma\in\Sigma_n$, we obtain
\begin{equation}
 \rho = \inf_{\substack{\sigma\in\Sigma_n \\ \text{s.t.} \int g \sigma \d\lambda=1}} \int f \sigma\d\lambda .
\end{equation}
To show the above equality, \todo{we use ???}.
This expression allows us to numerically estimate $\rho$ by solving a generalized eigenvalue problem on the moment matrix \todo{give more details here}. To do so, one needs to have access to the moments
$$
 m^f_\alpha = \int x^\alpha f(x)\d\lambda(x),
 \quad\text{and}\quad 
 m^g_\alpha = \int x^\alpha g(x)\d\lambda(x).
$$
\todo{I know I don't use the correct notation here...}
for any multi-index $\alpha\in\N^n$.
However, the performance of this approach degenerates when the dimension $n$ of the problem is large.
\todo{We have to strengthen this point: this is the actual motivation for using pushforward measures.}
}\fi
Doing so, we obtain the following hierarchy of SDP programs, indexed by $d\in\N$:
\begin{align}
\label{eq:rho_sos1}
\begin{array}{rl}
\rho_{d} =\displaystyle\inf_{\sigma\in \Sigma[x]_d}&  \, \displaystyle\int_{\saK} f(\x)\,\sigma(\x) \,\d \mu(\x): \\
\mbox{s.t.}&  \displaystyle\int_{\saK} g(\x) \, \sigma(\x)\, \d \mu(\x) \, = \, 1 \,. \\
\end{array}
\end{align}
The dual of \eqref{eq:rho_sos1} is given by
\begin{align}
\label{eq:rho_mom1}
\begin{array}{rl}
a_d=\displaystyle\sup_{a \in \R}& \,a: \\
\mbox{s.t.}& \M_{d}(f \, \y) \, \succeq \, a \, \M_d(g\, \y) \,.
\end{array}
\end{align}

We can now derive the rational function analog of Theorem \ref{th:cvg_pol}.
\begin{theorem}
\label{th:cvg_rat}
Let $\saK \subseteq \R^n$ be compact and $\mu \in \mathcal{M}(\saK)$ with moment sequence $\y$.
Let $f,g \in \R[\x]$, with $g$ positive on $\saK$.
Consider the hierarchy of semidefinite programs \eqref{eq:rho_mom1} indexed by $d\in\N$.
Then:
\begin{itemize}
\item[(i)] SDP \eqref{eq:rho_mom1} has an optimal solution $a_d \geq \rho$ for every $d\in\N$. 
\item[(ii)] There is no duality gap between SDP \eqref{eq:rho_sos1} and SDP \eqref{eq:rho_mom1}. Moreover, SDP \eqref{eq:rho_sos1} has an optimal solution $\sigma^\star \in \Sigma[\x]_{d}$ which satisfies $\int_{\saK} (f(\x) - a_d g(\x) \sigma^\star(\x) \d \mu(\x) = 0$. 
\item[(iii)] The sequence $(a_d)_{d \in \N}$ is monotone nonincreasing and $a_d \downarrow f^\star$ as $d\to \infty$.
\end{itemize}
\end{theorem}
\begin{proof}
The proof is similar to {that  of} Theorem \ref{th:cvg_pol}.\\
(i) First we prove that SDP \eqref{eq:rho_mom1} has always a feasible solution.
Recall that the rational function $\mathbf{x}\mapsto \frac{f(\mathbf{x})}{g(\mathbf{x})}$ admits a minimum $\rho$ on the compact set $\saK$ as a continuous function and because $g$ is positive. 
Let us take $a = \rho$. 
Then $f - a g$ is nonnegative on $\saK$. Hence, by Theorem \ref{th:loc}, one has $\M_d(f \, \y) - a \M_d( g\, \y) \, \succeq \, 0$. 
Next, we show that the value of any feasible point $a$ is bounded from above.
The first diagonal entry of the localizing matrix $\M_d((f -  a g) \, \y)$ is equal to $L_\y(f) - a L_\y(g)$.
By feasibility, one has $\M_d((f -  a g) \, \y) \succeq 0$, which implies that $L_\y(f) - a L_\y(g) \geq 0$.
Moreover, since $g$ is positive on $\saK$, Theorem \ref{th:loc} implies that $L_\y(g) > 0$. Thus $a < \frac{L_\y(f)}{L_{\y}(g)}$, the desired result.\\
(ii) As already proved in (i), any scalar $a < \rho$ is a feasible solution of SDP \eqref{eq:rho_mom1}. 
Since $f - a g >0$ on $\saK$, by Theorem \ref{th:loc} $a$ is a strictly feasible solution of SDP \eqref{eq:rho_mom1}. Therefore, Slater's condition  \cite[Theorem 3.1.]{Vandenberghe94SDP}  holds for SDP \eqref{eq:rho_mom1} and its dual SDP \eqref{eq:rho_sos1} admits an optimal solution $\sigma^\star$. 
Hence, there is no duality gap between SDP \eqref{eq:rho_sos1} and SDP \eqref{eq:rho_mom1}, and $\sigma^\star$ satisfies the desired equality.\\
(iii) One has $a_d \leq a_k$ for all $d \geq k$ since $\M_d((f -  a_d g) \, \y) \succeq 0$ implies that $\M_k((f -  a_d g) \, \y) \succeq 0$.
Therefore, the sequence $(a_d)_{d \in \N}$ is monotone nonincreasing. Since it is bounded from below by $\rho$, then it converges to $\rho^\star \geq \rho$.
\jean{Next, since $a_d \geq \rho^\star$ for all $d$, $\rho^\star$ is feasible for \eqref{eq:rho_mom1} for all $d$.} 
Then by Theorem \ref{th:loc}, $f -  \rho^\star g$ is nonnegative on $\saK$, and therefore  $\rho^\star\leq \rho$, the desired result.
\end{proof}

\oz{Notice that \eqref{eq:rho_mom1} and \eqref{eq:pol_mom} are essentially the same problem, the only difference being that $\M_d(\y)$ in  \eqref{eq:pol_mom} is replaced with $\M_d(g\,\y)$ in \eqref{eq:rho_mom1}. 
In order to circumvent the curse of dimensionality associated with these problems, we propose in the next section to reformulate the problem using with  pushforward measures \cite{Las19}.
}


\subsection{With the help of the pushforward measure}
\label{sec:pfm}
Let $\U: \saK \rightarrow \mathbb{R}^2$ be the function defined by
\begin{equation}
\U(\x):= 
\begin{pmatrix}
f(\mathbf{x})\\
g(\mathbf{x})
\end{pmatrix}=:\begin{pmatrix}
\U_1(\x)\\ \U_2(\x)
\end{pmatrix}.
\end{equation}
Replacing $f(\mathbf{x})$ by $u= \mathbf{U}_1(\mathbf{x})$ and $g(\mathbf{x})$ by $v= \U_2(\mathbf{x})$ in \eqref{eq:minfg} \jean{yields:}
$$
 \rho = \min_{(u,v)\in \U(\saK)} \frac{u}{v}\,,
$$
\jean{and by Proposition \ref{prop:minfg_measure}:}
\begin{align}
\label{eq:minfg_pfm}
\begin{array}{rl}
\rho =\displaystyle\inf_{\mu \in \mathcal{M}(\U(\saK))}&  \, \displaystyle\int_\saK u \,\d \mu(u,v) \\
\mbox{s.t.}&  \displaystyle\int_\saK v \, \d \mu(u,v) \, = \, 1 \,,
\end{array}
\end{align}
where the variable $\mu$ is now a measure supported on the $2$-dimensional image set $\U(\saK) \subset\R^2$.
Now \jean{fix the reference measure $\mu := \U_\# \lambda$, the pushforward  of Lebesgue measure $\lambda$ on $\saK$, by the mapping $\U$.
By construction its support is $\U(\saK)$.
%
%
Then consider the following SDP \jean{indexed by} $d\in\N$:}
\begin{align}
\label{eq:rho_pfm_sos1}
\begin{array}{rl}
\rho_{d}^{\#}=\displaystyle\inf_{\sigma}& \,\displaystyle\int_\saK u\,\sigma(u,v)\,\d \U_\# \lambda(u,v) \\
\mbox{s.t.}& \displaystyle\int_\saK v \, \sigma(u,v)\,\d\U_\# \lambda(u,v)\,=\,1\\
&\sigma\in\Sigma[u,v]_d\,.
\end{array}
\end{align}
The dual of SDP \eqref{eq:rho_pfm_sos1} is
\begin{align}
\label{eq:rho_pfm_mom1}
\begin{array}{rl}
a_d^\# = \displaystyle\sup_{a \in\R}& \,a \\
\mbox{s.t.}& \M_d(u \, \y^\#) \, \succeq \, a \, \M_d(v\, \y^\#)  \,.
\end{array}
\end{align}
\if{
\begin{align}
 \rho &=
 \inf_{\substack{\sigma\in\Sigma_2 \\ \text{s.t.} \int z_2 \sigma(\mathbf{z}) \d(\mathbf{T}_\sharp\lambda)=1}} \int z_1 \sigma(\mathbf{z}) \d(\mathbf{T}_\sharp\lambda) \\
 &= \inf_{\substack{\sigma\in\Sigma_2 \\ \text{s.t.} \int g(\mathbf{x}) \sigma(\mathbf{T}(\mathbf{x})) \d\lambda=1}} \int f(\mathbf{x}) \sigma(\mathbf{T}(\mathbf{x})) \d\lambda
\end{align}
}\fi
\jean{For $d$ fixed, arbitrary, solving \eqref{eq:rho_pfm_mom1} numerically first requires computing the entries
of matrices $\M_d(u \, \y^\#)$ and $\M_d(v\, \y^\#)$, that is, computing the moments:}
\begin{equation}
    \label{eq:require}
  y_{i,j}^{\#} := \int_{\R^2} u^i v^j\, \d\U_\# \lambda(u,v) = \int_{\saK} f(\mathbf{x})^{i}g(\mathbf{x})^{j}\,\d\lambda(\mathbf{x}) \,,\quad \forall (i,j)\in\N^2_{2d}\,.
\end{equation}
\jean{Then, once this is done, 
solving \eqref{eq:rho_pfm_mom1} is a relatively easy generalized eigenvalue problem.}
\jean{In general, computing integrals 
\eqref{eq:require} is out of reach for 
arbitrary semialgebraic sets. However, if the set $\saK$  is ``simple'', e.g., the unit box $[-1,1]^n$, the unit sphere $\bS^{n-1}$ (or their  image by any affine mapping), then this computation becomes simpler and can be done exactly and in closed form.}

\begin{theorem}
\label{th:cvg_rat_pfm}
Let $\saK \subseteq \R^n$ be compact.
Let $f,g \in \R[\x]$, such that $g$ is positive on $\saK$.
Consider the hierarchy of semidefinite programs \eqref{eq:rho_sos1} indexed by $d\in\N$.
Then:
\begin{itemize}
\item[(i)] SDP \eqref{eq:rho_pfm_mom1} has an optimal solution $a_d^\# \geq \rho$ for every $d\in\N$.
\item[(ii)] There is no duality gap between SDP \eqref{eq:rho_pfm_sos1} and SDP \eqref{eq:rho_pfm_mom1}, and SDP \eqref{eq:rho_pfm_sos1} has an optimal solution $\sigma^\star \in \Sigma[u,v]_{d}$. 
\item[(iii)] The sequence $(a_d^\#)_{d \in \N}$ is monotone nonincreasing and $a_d^\# \downarrow \rho$ as $d\to \infty$.
\end{itemize}
\end{theorem}
\begin{proof}
This is direct application of Theorem \ref{th:cvg_rat}  with the notation $\saK \leftarrow  \U(\saK)$, $\mu \leftarrow \d\U_\# \lambda$, $n \leftarrow m$, $f \leftarrow u$ and $g \leftarrow v$.
\end{proof}

So as soon as the information \eqref{eq:require} required to
fill up the entries of the two matrices $\M_d(u\,\y^\#)$ and $\M_d(v\,\y^\#)$ is available, solving \eqref{eq:rho_pfm_mom1} is a relatively easy generalized eigenvalue problem \jean{ for two real symmetric matrices of reasonable $O(d^2)$ size.} 

However even though integrals \eqref{eq:require} can be computed exactly in closed-form for simple sets, they can be very tedious to compute if one has to expand the integrand in the monomial basis. \jean{As already mentioned in the introduction,  even for moderate dimensions $(n,d)$, computing efficiently \eqref{eq:require} is a scientific challenge of its own, with dramatic consequences for 
solving \eqref{eq:rho_pfm_mom1}. For instance, if $\saK$ is a simplex then 
some efficient methods are described in in \cite{baldoni} and in particular 
when the degree is fixed, \eqref{eq:require} is a tractable problem.}

\section{Minimizing a sum of rational functions}
\label{sec:multirat}
This section focuses on to the case where there are potentially several terms in the sum given in \eqref{quotients}, i.e., the case where $N \geq 1$. \jean{Of course by reducing to same common denominator,
the problem reduces to minimlizing a single fraction. However the degree in 
both numerator and denominator can be too high and a specific approach is needed. The one that we propose
is similar in spirit to that in \cite{bugarin2016minimizing} 
for computing lower bounds.}

\subsection{An instance of the generalized moment problem}
\label{sec:gpmN}

We start \jean{by recalling a} result stated in Theorem 2.1 of \cite{bugarin2016minimizing}. 
The problem of computing $\rho$ can be cast as a particular instance of the generalized moment problem (GMP), namely 
\begin{align}
\label{eq:rho}
\begin{array}{rl}
\rho=\displaystyle\min_{\mu_i}& \left\{\,\displaystyle\sum_{i=1}^N \int_\saK f_i(\x)\,\d \mu_i(\x): \right.\\
\mbox{s.t.}&\displaystyle\int_\saK \x^\alpha\,g_i(\x)\,\d \mu_i(\x)\,=\,\displaystyle\int_\saK \x^\alpha\,g_1(\x)\,\d\mu_1(\x) \,, \: 
\alpha \in \N^n   \,,  \: 1<i\leq N\\
&\displaystyle\int_\saK g_1(\x) \, \d \mu_1(\x) \, = \, 1 \,, \\
&\left.\mu_1,\dots,\mu_N\in\mathcal{M}(\saK)\,
\right\}.
\end{array}
\end{align}
{Inspired by the single fraction case in \S\ref{sec:singlerat}, one can obtain an approximation for $\rho$} by restricting each measure $\mu_i \in \mathcal{M}(\bS^{n-1})$ to be absolutely continuous w.r.t. the reference measure (e.g., the Lebesgue measure on $\bS^{n-1}$) with an SOS density of degree $2 d$, \jean{and considering  moment equality constraints only up to order $s$. So consider the semidefinite program:}
\begin{align}
\label{eq:rho_sos}
\begin{array}{rl}
\rho_{d,s} =\displaystyle\inf_{\sigma_i}& \left\{\,\displaystyle\sum_{i=1}^N \int_\saK f_i(\x)\,\sigma_i(\x) \,\d \lambda: \right.\\
\mbox{s.t.}&  \displaystyle\int_\saK \x^\alpha\,g_i(\x)\,\sigma_i(\x) \,\d \lambda\,=\,\displaystyle\int_\saK \x^\alpha\,g_1(\x)\,\sigma_1(\x)\, \d \lambda   \,, \: 
\vert\alpha\vert\,\leq s  \,,  \:1<i\leq N\\
& \displaystyle\int_\saK g_1(\x) \, \sigma_1(\x)\, \d \lambda \, = \, 1 \,, \\
&\left.\sigma_1,\dots,\sigma_N\in\Sigma[\x]_{d}\,
\right\}.
\end{array}
\end{align}
The dual of \eqref{eq:rho_sos} is given by
\begin{align}
\label{eq:rho_mom}
\begin{array}{rl}
\displaystyle\sup_{a,h_i}& \left\{\,a: \right.\\
\mbox{s.t.}& \M_{d}(f_1 \, \y) \, \succeq \, a \, \M_d(g_1\, \y) + \displaystyle\sum_{i=2}^N  \M_d(h_i \, g_1 \, \y)  \,,\\
& \M_d(f_i \, \y) + \M_d(h_i \, g_i\, \y) \, \succeq \, 0 \,, \:1<i\leq N \,, \\
&\left. a\in \R,h_2,\dots,h_N\in\R[\x]_{s}\,
\right\}.
\end{array}
\end{align}
\jean{
\begin{proposition}
\label{rk:feas}
SDP \eqref{eq:rho_mom} has a feasible solution for all $(d,s)$ while
SDP \eqref{eq:rho_sos} has a feasible solution for sufficiently large $d$.
\end{proposition}}
\begin{proof}
\jean{For each $i=1,\dots,N$, and as $g_i$ is positive on the compact set $\saK$, the rational function $\frac{f_i}{g_i}$ is continuous and thus has a minimum $a_i$ on $\saK$.
Next, consider the constant polynomial $h_i = -a_i$ and $a = \sum_{i=1}^N a_i$.
Then $f_i + h_i g_i$ is nonnegative on $\saK$ and by Theorem \ref{th:loc}, $\M_d(f_i \, \y) + \M_d(h_i \, g_i\, \y) \, \succeq \, 0$.
In addition, 
$$\frac{f_1}{g_1} \geq a_1 = a - \sum_{i=2}^N a_i = a + \sum_{i=2}^N h_i,$$
which implies $f_1 \geq a g_1 + \sum_{i=2}^{N} h_i g_1$ because $g_1$ is positive.
Again by Theorem \ref{th:loc}, $\M_d(f_1 \, \y) \, \succeq \, a \, \M_d(g_1\, \y) + \sum_{i=2}^N  \M_d(h_i \, g_1 \, \y)$,
which proves that $(a,h_2,\dots,h_N)$ is a feasible solution for SDP \eqref{eq:rho_mom}.}
However, note that SDP \eqref{eq:rho_sos} may not have a feasible solution for given $d,s \in \N$ even if the infinite-dimensional LP \eqref{eq:rho} has one. 
\jean{But SDP  \eqref{eq:rho_sos} indeed has a feasible solution provided that $d$ is sufficienly large.} Since each $g_i$ is globally positive,  by Hilbert-Artin's representation, there exist SOS polynomials $p_i$ and $q_i$ such that $g_i = \frac{p_i}{q_i}$.
\jean{Define 
\[
C := \int_{\saK} \prod_{j=1}^N p_i(\x) \, \d \lambda(\x)\quad 
\mbox{and}\quad 
\sigma_i :=  \frac{q_i}{C}  \prod_{\substack{1 \leq j \leq N  \\ j \neq i}} p_i \,,
\]
and let $d_i:={\rm deg}(\sigma_i) = \deg q_i + \displaystyle\sum_{\substack{1 \leq j \leq N  \\ j \neq i}} \deg p_j$.}

\jean{Each SOS polynomial $\sigma_i$ satisfies $g_i \sigma_i = \frac{1}{C}  \prod_{j=1}^N p_i$ and $\displaystyle\int_\saK g_1(\x) \sigma_1(\x) \, \d \lambda(\x) = 1$, and thus $(\sigma_1,\dots,\sigma_N)$ is feasible for SDP  \eqref{eq:rho_sos}
whenever $d \geq \max_{1\leq i \leq N}  d_i $.}
\end{proof}
To overcome the feasibility issue of SDP \eqref{eq:rho_sos}, one remedy is to allow one an $\varepsilon$-violation of the equality constraints with $\varepsilon > 0$. \jean{Doing so \eqref{eq:rho_sos} now reads:}

\begin{align}
\label{eq:rho_sos_eps}
\begin{array}{rl}
\rho_{s} (\varepsilon) =\displaystyle\inf_{\sigma_i}& \left\{\,\displaystyle\sum_{i=1}^N \int_\saK f_i(\x)\,\sigma_i(\x) \,\d \lambda(\x): \right.\\
\mbox{s.t.}& \bigg| \displaystyle\int_\saK \x^\alpha\,g_i(\x)\,\sigma_i(\x) \,\d \lambda(\x)\,-\,\displaystyle\int_{\saK} \x^\alpha\,g_1(\x)\,\sigma_1(\x)\, \d \lambda(\x) \bigg| \leq \varepsilon  \,, \\ 
&\vert\alpha\vert\,\leq s  \,,  \:1<i\leq N\\
& \bigg| \displaystyle\int g_1 \, \sigma_1\, \d \lambda \, - \, 1 \bigg| \leq \varepsilon \,, \\
&\left.\sigma_1,\dots,\sigma_N\in\Sigma[\x]_{d}\,
\right\},
\end{array}
\end{align}
with $d \geq d_0(s)$, for a suitably well-chosen (in particular large enough) $d_0(s)$. 
In order to prove convergence of the approximation bounds $\rho_{s}(\varepsilon)$ to $\rho$, when $s \to \infty$ and $\varepsilon \to 0$, we need to rely on specific approximation results
\jean{provided in  \cite{de2020convergence,laurent2020near,slot2020improved} which depend on the choice of the reference measure $\mu$ and the set $\saK$;} see \cite[Table 2]{de2019survey} as well as \cite{de2020convergence,laurent2020near,slot2020improved} for more details.

{In the sequel we illustrate the approach for the particular case} when $\mu$ is the Lebesgue measure on the unit cube $\saK = [-1, 1]^n$. 
We \jean{need} the following auxiliary approximation result, easily derived from Corollary 2 in \cite{slot2020improved}.
\begin{theorem}
\label{th:diracSOS}
Let $f$ be a polynomial with global minimizer $\x^\star$ on the unit cube $\saK = [-1,1]^n$. There exist a sequence of SOS polynomials $(\sigma^{(d)})_{d}$ such that $\sigma^{(d)}$ is of degree $2d$, does not depend on the degree of $f$, and satisfies 
\begin{align}
\label{eq:diracSOS}
\int_{\saK} \sigma^{(d)}(\x) \d \lambda(\x) = 1 \,, \quad 
\int_{\saK} f(\x) \, \sigma^{(d)}(\x) \d \lambda(\x) - f(\x^\star) \leq \frac{C(f)}{d^2} \,,
\end{align}
\jean{for all $d$, and} where $C(f)$ depends on the degree of $f$, $n$.
\end{theorem}
\begin{theorem}
\label{th:cvgN}
Assume that each denominator $g_i$ takes only positive values on $\saK$, for each $i=1,\dots,N$.
Let us fix $d,s \in \N$ and define the following parameter:
\begin{align*}
\Delta_{d,s} := \min_{\sigma_i \in \Sigma[\x]_d} \max & \Bigg\lbrace \max_{\substack{|\alpha|\leq s \\ i = 2,\dots,N}}  \left\lbrace \bigg| \displaystyle\int \x^\alpha\,g_i\,\sigma_i \,\d \lambda\,-\,\displaystyle\int \x^\alpha\,g_1\,\sigma_1\, \d \lambda \bigg| \right\rbrace,  
\bigg| \displaystyle\int g_1 \, \sigma_1\, \d \lambda \, - \, 1 \bigg|, \\
& \bigg| \sum_{i=1}^N \int f_i\,\sigma_i \,\d \lambda - \rho \bigg|
 \Bigg\rbrace \,.\\
\end{align*}
Then, there exists a sequence of positive integers $(d_0(s))_{s\in\N}$ such that
\jean{\begin{equation}
   \label{eq:conv-delta} 
\lim_{s \to \infty} \Delta_{d_0(s),s} = 0 \quad \text{and} \quad  \lim_{\substack{s \to \infty}} \rho_{s}\left(\frac{1}{\sqrt{d_0(s)}}\right) = \rho \,.
\end{equation}}
\end{theorem}
\begin{proof}
Let $\x^\star$ be a global minimizer of the sum of rational fractions $\sum_{i=1}^N \frac{f_i}{g_i}$ on the unit cube $\saK$.
The proof follows the same line of reasoning as \cite[Theorem 7]{de2019survey}.
Let us consider for each $i=1,\dots,N$ the polynomial 
\begin{align}
\label{eq:pis}
p_{i,s}(\x) := (f_i(\x)-f_i(\x^\star))^2 +  \sum_{|\alpha|\leq s} (\x^{\alpha} g_i(\x) - {\x^{\star}}^{\alpha} g_i(\x^\star))^2 \,,
\end{align}
which minimal value over $\saK$ is equal to 0 and attained at $x^{\star}$.
Then, by Theorem \ref{th:diracSOS}, there exist a sequence of polynomials $({\sigma'}_i^{(d)})_d$ such that for all $d$,
\begin{align}
\label{eq:diracSOS2}
\int_{\saK} {\sigma'}_i^{(d)}(\x) \d \lambda = 1 \,, \quad 
\int_{\saK} p_{i,s}(\x) \, {\sigma'}_i^{(d)}(\x) \d \lambda \leq \frac{C_{i,s}}{d^2} \,,
\end{align}
where $C_{i,s}$ depends on the degrees of $p_{i,s}$, $n$, $\|\nabla p_{i,s} (x^\star)\|_2$ and $\max_{\x \in \saK} \|\nabla^2 p_{i,s}(\x) \|_2$.
Let us define $\sigma_i^{(d)}(\x) := g_i(\x^\star)^{-1}  {\sigma'}_i^{(d)}(\x)$. 
By assumption, $g_i$ takes only positive values on $\saK$ thus  $\sigma_i^{(d)}$ is a well defined SOS polynomial.
Then, one has for all $\alpha$ satisfying $|\alpha| \leq s$:
\begin{align*}
\bigg| \int_{\saK} \x^{\alpha} g_i(\x) {\sigma}_i^{(d)}(\x) \d \lambda(\x) - {\x^{\star}}^{\alpha}  \bigg|^2  & = g_i(\x^\star)^{-2} \bigg| \int_{\saK} (\x^{\alpha} g_i(\x) - {\x^{\star}}^{\alpha} g_i(\x^\star)) {\sigma'}_i^{(d)}(\x) \d \lambda(\x)  \bigg|^2 \\
& \leq g_i(\x^\star)^{-2} \int_{\saK} (\x^{\alpha} g_i(\x) - {\x^{\star}}^{\alpha} g_i(\x^\star))^2 {\sigma'}_i^{(d)}(\x) \d \lambda(\x)   \\
& \leq g_i(\x^\star)^{-2} \int_{\saK} p_{i,s}(\x) \, {\sigma'}_i^{(d)}(\x) \d \lambda(\x) \leq \frac{g_i(\x^\star)^{-2} C_{i,s} }{d^2} \,,
\end{align*}
where we used the equality from \eqref{eq:diracSOS2} to obtain the first equality, Jensen's inequality to derive the first inequality as well as the inequality from \eqref{eq:diracSOS2} to obtain the last inequality.
This implies that for all $d\in\mathbb N$ one has
\[ 
\bigg| \int_{\saK} \x^{\alpha} g_i(\x) {\sigma}_i^{(d)}(\x) \d \lambda(\x) - {\x^{\star}}^{\alpha}  \bigg| \leq \frac{g_i(\x^\star)^{-1} C_{i,s}^{1/2} }{d} \,.
\]
As a direct consequence, we obtain for $\alpha = 0$ and $i=1$:
\[ 
\bigg| \int_{\saK}  g_1(\x) {\sigma}_1^{(d)}(\x) \d \lambda(\x) - 1  \bigg| \leq \frac{g_1(\x^\star)^{-1} C_{1,s}^{1/2} }{d} \,.
\]
and by using the triangular inequality
\[
\bigg| \int_{\saK} \x^{\alpha} g_i(\x) {\sigma}_i^{(d)}(\x) \d \lambda(\x) - \int_{\saK} \x^{\alpha} g_1(\x) {\sigma}_1^{(d)}(\x) \d \lambda(\x)  \bigg| \leq \frac{g_i(\x^\star)^{-1} C_{i,s}^{1/2} + g_1(\x^\star)^{-1} C_{1,s}^{1/2} }{d} \,.
\]
Similarly, one proves that
\[
\bigg| \int_{\saK} f_i(\x) {\sigma}_i^{(d)}(\x) \d \lambda(\x) - \frac{f_i(\x^\star)}{g_i(\x^\star)}  \bigg| \leq \frac{g_i(\x^\star)^{-1} C_{i,s}^{1/2} }{d} \,,
\]
which implies
\[
\bigg| \sum_{i=1}^N \int_{\saK} f_i\,\sigma_i^{(d)} \,\d \lambda(\x) - \rho \bigg| \leq \frac{\sum_{i=1}^N g_i(\x^\star)^{-1} C_{i,s}^{1/2} }{d} \,.
\]
By selecting for any $s \in \N$
\[
d_0(s) := \max_{1\leq i \leq N} \bigg\lbrace \left( g_i(\x^\star)^{-1} C_{i,s}^{1/2} + g_1(\x^\star)^{-1} C_{1,s}^{1/2} \right)^2 \bigg\rbrace \,,
\]
one has for all $d\in\mathbb{N}$ satisfying $d \geq d_0(s)$
\[
\bigg| \int_{\saK} \x^{\alpha} g_i(\x) {\sigma}_i^{(d)}(\x) \d \lambda(\x) - \int_{\saK} \x^{\alpha} g_1(\x) {\sigma}_1^{(d)}(\x) \d \lambda(\x)  \bigg| \leq \frac{1 }{\sqrt{d_0(s)}} \,,
\]
yielding the desired convergence \jean{result \eqref{eq:conv-delta}.}
\end{proof}
\begin{remark}
Note that \jean{with $s$ fixed}, the sequence  $(\rho_{d,s})_{d\in\N}$ is monotone nonincreasing.
Indeed when one fixes the degree of each $h_i$ and optimizes only over $a$, the value of the supremum decreases when $d$ increases, i.e., $\rho_{d+1,s} \leq \rho_{d,s}$.
Besides, for each $d$, $(\rho_{d,s})_{s\in\N}$ is monotone nondecreasing. 
When one fixes the order of the localizing matrices, and optimizes over polynomials $h_i$ of increasing degrees, the value of the supremum increases since the feasible set becomes larger, i.e., $\rho_{d,s} \leq \rho_{d,s+1}$.
However, there is no systematic dominance relation between $\rho_{d+1,s+1}$ and $\rho_{d,s}$.
\end{remark}
For the general setting of a \jean{compact} set $\saK \subset \R^n$ 
and a fixed arbitrary Borel measure $\mu \in \mathcal{M}(\saK)$, one can rely on the following corollary of Theorem \ref{th:cvg_pol}:
\begin{corollary}
\label{cor:diracSOS}
Let $f$ be a polynomial with global minimizer $\x^\star$ on the compact set $\saK$. 
There exists a sequence of SOS polynomial $(\sigma^{(d)}_{d\in\N}$ of degree $2d$ and a function $c:\N \to [0,+\infty)$ such that
\begin{align}
\label{eq:diracSOSgen}
\int_{\saK} \sigma^{(d)}(\x) \d \lambda(\x) = 1 \,, \quad 
\int_{\saK} f(\x) \, \sigma^{(d)}(\x) \d \lambda(\x) - f(\x^\star) \leq c(d) \,,
\end{align}
and $c(d) \to 0$ as $d \to \infty$.
\end{corollary}
We briefly outline how to adapt the convergence proof of Theorem \ref{th:cvgN}  for the general setting.
\jean{Invoking} Corollary \ref{cor:diracSOS}, one obtains a function $c_{i,s} : \N \to [0, +\infty)$ associated to the polynomial $p_{i,s}$ defined in \eqref{eq:pis} with $c_{i,s}(d_0(s)) \to 0$ as $s\to\infty$. 
Then one can replace $\frac{C_{i,s}}{d^2}$ by $c_{i,s}(d)$ in the proof of Theorem \ref{th:cvgN}.

\subsection{With the help of pushforward measures}
\label{sec:pfmN}
\jean{To approximate $\rho$ in \eqref{eq:rho} we follow
and a methodology similar to that in the univariate case.}
Consider the mapping $\U:\saK\to\R^{2N}$ defined by:
\begin{equation}
    \label{eq:mapping1}
\x\mapsto \U(\x)=\begin{bmatrix} f_1(\x)  & g_1(\x) & \ldots & f_N(\x) & g_N(\x)\end{bmatrix}^\top \,.
\end{equation}
Let $\u = (u_1,\dots,u_N)$, $\v = (v_1,\dots,v_N) \in \R^N$ and let $\U_\# \lambda (\u,\v)$ be the pushforward of the restriction of the uniform measure $\lambda$ on $\saK$ with respect to $\U$.
Then
\[\rho=\min_{\x\in \saK} \sum_{i=1}^N\frac{f_i(\x)}{g_i(\x)} \, = \, 
\min_{(\u,\v)\,\in \U(\saK) }\,\sum_{i=1}^n\frac{u_i}{v_i}.\]
Then for every $d,s\in\N$ define:
\begin{align}
\label{eq:rho_pfm_sos}
\begin{array}{rl}
\rho_{d,s}^{\#}=\displaystyle\min_{\sigma_i}& \left\{\,\displaystyle\sum_{i=1}^N \int_\saK u_i\,\sigma_1(\u,\v)\,\d \U_\#\lambda(\u,\v) : \right.\\
\mbox{s.t.}&\displaystyle\int_\saK \u^\alpha\v^\beta\,v_i\,\sigma_i(\u,\v)\,\d \U_\#\lambda(\u,\v)  =\,\displaystyle\int_\saK \u^\alpha\v^\beta\, v_1\,\sigma_1(\u,\v) \, \d\U_\#\lambda(\u,\v)\,,\\ 
&\vert\alpha+\beta\vert\,\leq s  \,,\: 1<i\leq N\\
&\displaystyle\int_\saK u_1 \, \sigma_1(\u,\v)\,\d\U_\#\lambda(\u,\v)\,=\,1\\
&\left.\sigma_i\in\Sigma[\u,\v]_d\,
\right\}.
\end{array}
\end{align}
The dual of SDP \eqref{eq:rho_pfm_sos} is
\begin{align}
\label{eq:rho_pfm_mom}
\begin{array}{rl}
\displaystyle\sup_{a,h_i}& \left\{\,a: \right.\\
\mbox{s.t.}& \M_d(u_1 \, \y^\#) \, \succeq \, a \, \M_d(v_1\, \y^\#) + \sum_{i=2}^N  \M_d(h_i \, u_1 \, \y^\#)  \,,\\
& \M_d(u_i \, \y^\#) + \M_d(h_i \, v_i\, \y^\#) \, \succeq \, 0 \,, \:1<i\leq N \,, \\
&\left. a\in \R,h_2,\dots,h_N\in\R[\u,\v]_{s}\,
\right\}.
\end{array}
\end{align}
In order to obtain a convergence result 
\jean{as in Theorem \ref{th:cvgN}, we could try to invoke 
specific approximation results} for the reference measure $\U_\# \lambda$. \jean{However and unfortunately}, the only known result is in the polynomial optimization case \cite{laurent2020near}, when $\U = f$.
For every $s\in\N$ define:
\begin{align}
\label{eq:rho_pfm_sos_eps}
\begin{array}{rl}
\rho_{s}^\# (\varepsilon) =\displaystyle\inf_{\sigma_i}& \left\{\,\displaystyle\sum_{i=1}^N \int_\saK u_i(\x)\,\sigma_i(\u,\v) \,\d \U_\# \lambda(\u,\v): \right.\\
\mbox{s.t.}& \bigg| \displaystyle\int_\saK \u^\alpha \v^{\beta}\,v_i\,\sigma_i(\u,\v) \,\d \U_\# \lambda(\u,\v)\,-\,\displaystyle\int_\saK \u^\alpha \v^{\beta}\,v_1\,\sigma_1(\u,\v)\, \d \U_\# \lambda(\u,\v) \bigg| \leq \varepsilon  \,, \\
&\vert\alpha+\beta\vert\,\leq s  \,,  \:1<i\leq N\\
& \bigg| \displaystyle\int_{\saK} v_1 \, \sigma_1(\u,\v)\, \d \U_\# \lambda(\u,\v) \, - \, 1 \bigg| \leq \varepsilon \,, \\
&\left.\sigma_1,\dots,\sigma_N\in\Sigma[\u,\v]_{d}\,
\right\},
\end{array}
\end{align}
with $d \geq d_0(s)$, for a well-chosen $d_0(s)$. 
The convergence of $\rho_{s}^\# (\varepsilon) \to \rho$ as $s\to\infty$ and $\varepsilon \to 0$ is obtained as in Section \ref{sec:gpmN}, by using Corollary \ref{cor:diracSOS} together with the polynomial $q_{i,s}(\u,\v) := (u_i-f_i(\x^\star))^2 +  \sum_{|\alpha+\beta|\leq s} (\u^{\alpha} \v^{\beta} v_i - {\x^{\star}}^{\alpha} g_i(\x^\star))^2 $ (which plays the same role as the polynomial $p_{i,s}$ in the proof of Theorem \ref{th:cvgN}).

An alternative framework to minimize a sum of fractions is presented in Appendix \ref{sec:multirat_hierarchy2}, together with promising  investigation tracks.
\if{
OLD BMI Formulation:
\[
\begin{array}{rl}
\rho_d=\displaystyle\min_{\sigma,\phi_1,\ldots,\phi_N}& \left\{\,\displaystyle\sum_{i=1}^N \int u_i\,\sigma\,d\phi_i: \right.\\
\mbox{s.t.}&\displaystyle\int \u^\alpha\v^\beta\,v_i\,d\phi_i\,=\,\displaystyle\int \u^\alpha\v^\beta\,d\U_\star(\saK),\:i\leq N\\
&\displaystyle\int\sigma\,d\U_\star(\bS^{n-1})\,=\,1\\
&\left.\sigma\in\Sigma[\u,\v]_d\,
\right\}.
\end{array}\]
\begin{theorem}
Assume that $g_i>0$ (or $g_i<0$) for all $\x\in\mathbb{S}^{n-1}$. Then
$\rho_d\downarrow \rho$ as $d\to\infty$.
\end{theorem}

It suffices to observe that
$$\rho=\displaystyle\inf_{\sigma\in\Sigma[\u,\v]}\left\{\int \sum_{i=1}^N\frac{u_i}{v_i}\sigma\,d\U_\star(\bS^{n-1}):\:\int\sigma\,d\U_\star(\bS^{n-1})=1\right\},$$
and if
\[\int\u^\alpha\,\v^\beta\,v_i\,d\phi_i\,=\,\int\u^\alpha\,\v^\beta\,d\U_\star(\bS^{n-1}),\quad \forall (\alpha,\beta),\]
with $i=1,\ldots,N,$ then $v_i\,d\phi_i=d\U_\star(\bS^{n-1})$ for all $i=1,\ldots N$, and therefore:
\begin{eqnarray*}
\displaystyle\sum_{i=1}^n\,\displaystyle\int u_i\,\sigma\,d\phi_i&=&
\displaystyle\sum_{i=1}^n\int\frac{u_i}{v_i}\,\sigma\,v_i\,d\phi_i\\
&=&\displaystyle\sum_{i=1}^n\displaystyle\int \frac{u_i}{v_i}\,\sigma\,d\U_\star(\bS^{n-1})
\end{eqnarray*}
Such an algorithm works well if $N$ is relatively small, and if one knows how to compute or approximate well
\[\displaystyle\int_{\mathbb{S}^{n-1}}\prod_{i=1}^N f_i(\x)^{\alpha_i}g_i(\x)^{\beta_i}\,d\mu(\x)\,=:\,\int \u^\alpha\,\v^\beta\,d\U_\star(\bS^{n-1}) \,,\]
where $\mu$ is the rotation invariant measure on $\mathbb{S}^{n-1}$ (at least for 
all $\vert\alpha+\beta\vert\leq 2d$). 
One can show that this latter optimization problem reduces again to a generalized eigenvalue problem.
}\fi
\section{Numerical experiments}
\label{sec:benchs}
Here, we illustrate our theoretical framework for the minimization of rational functions on the hypercube $[-1,1]^n$ with \victor{a set of preliminary numerical experiments.
Our experiments are performed with Julia, and we rely on JuMP \cite{jump} and Mosek~\cite{mosek} to model and solve SDP problems, respectively. 
All results were obtained on an Intel Xeon(R) E-2176M CPU (2.70GHz $\times$ 12) with 32Gb of RAM.
Our code is available online\footnote{\url{http://homepages.laas.fr/vmagron/files/rational_pfm.zip}}.}

\jean{We emphasize that the main goal of these experiments is to illustrate that
the pushforward approach yields better results than the usual (multivariate) second Lasserre's hierarchy of upper bounds. 
However most of the computational burden is spent in solving several problems of the form \eqref{eq:integral}
by doing a naive and costly expansion of the integrand in the monomial basis. Therefore in its present form 
this approach is not competitive with standard local optimization algorithms starting from various initial points.
But in our opinion, these numerical experiments provide an incentive to further study
efficient algorithms for computing integrals of the form \eqref{eq:integral} even for a restricted class of polynomials $f$ and sets $\saK$. 
In doing so one could solve higher steps of the hierarchy to get better approximations of the minimum, and address larger size problems as well. For instance if $\saK$ is a simplex then efficient methods described in \cite{baldoni} can be exploited.}

\subsection{Single rational functions}
Here, we focus on the case $N=1$ by considering instances of the minimization problem \eqref{eq:minfg}.
We compare the values of $\rho_d$ and $\rho_d^\#$ when solving SDP \eqref{eq:rho_sos1} and \eqref{eq:rho_pfm_sos1}, respectively, for increasing values of $d \in \N$, as well as the timings needed to obtain them.
Each reported timing is in seconds and includes both the time required to compute  the entries of the SDP matrices (i.e., the time spent to compute the integrals of the monomials on $\saK$) and the solving time of the resulting SDP problem.\\
First, we consider the fraction $\frac{f}{g}$ from Example \ref{ex1}. 
\begin{example}
Let us take
\label{ex1}
\begin{align*}
f(\x) = \sum_{i=1}^n x_i^{2 n} \,, \quad g(\x) = \prod_{i=1}^n x_i^2 \,, \quad  \saK= [-1,1]^n \,.
\end{align*}
Note that $\rho = \min_{\x \in \saK} \frac{f(\x)}{g(\x)} = n$ as a consequence of the inequality of arithmetic and geometric means.
\end{example}

\begin{table}[!ht]
\caption{Upper bounds obtained for the minimum of the function of Example \ref{ex1}.}
\label{table:ex1}
\begin{center}
\begin{tabular}{|c|c|cc|cc|}
\hline
\multirow{2}*{$n$}  & \multirow{2}*{$d$}   & \multicolumn{2}{c|}{$\rho_d$} &  \multicolumn{2}{c|}{$\rho_d^\#$}   \\
 &    &     value & time & value & time\\
 \hline
\multirow{8}*{2} & 1  & 3.15 & 0.01 & 2.16 & 0.01 \\ 
 & 2  & 2.37 & 0.01 & 2.04 & 0.02 \\ 
 & 3  & 2.21 & 0.01 & 2.02 & 0.02 \\ 
 & 4  & 2.11 & 0.01 & 2.01 & 0.02 \\ 
 & 5  & 2.07 & 0.02 & 2.01 & 0.04 \\ 
 & 6  & 2.05 & 0.12 & 2.01 & 0.08 \\ 
 & 7  & 2.03 & 0.22 & 2.01 & 0.18 \\ 
 & 8  & 2.02 & 0.42 & 2.01 & 0.35 \\ 
\hline
\end{tabular}
\begin{tabular}{|c|c|cc|cc|}
\hline
\multirow{2}*{$n$}  & \multirow{2}*{$d$}   & \multicolumn{2}{c|}{$\rho_d$} &  \multicolumn{2}{c|}{$\rho_d^\#$}   \\
 &    &     value & time & value & time\\
 \hline
\multirow{8}*{3} & 1  & 9.29 & 0.01 & 3.66 & 0.01 \\ 
 & 2  & 5.45 & 0.01 & 3.19 & 0.05 \\ 
 & 3  & 4.63 & 0.02 & 3.08 & 0.06 \\ 
 & 4  & 3.85 & 0.09 & 3.05 & 0.07 \\ 
 & 5  & 3.60 & 0.62 & 3.02 & 0.10 \\ 
 & 6  & 3.36 & 4.96 & 3.02 & 0.14 \\ 
 & 7  & 3.27 & 23.1 & 3.01 & 0.21 \\ 
 & 8  & 3.19 & 156. & 3.01 & 0.48 \\ 
\hline
\end{tabular}
\\
\begin{tabular}{|c|c|cc|cc|}
\hline
\multirow{2}*{$n$}  & \multirow{2}*{$d$}   & \multicolumn{2}{c|}{$\rho_d$} &  \multicolumn{2}{c|}{$\rho_d^\#$}   \\
 &    &     value & time & value & time\\
 \hline
\multirow{8}*{4} & 1  & 27.3 & 0.01 & 5.75 & 0.02 \\ 
 & 2  & 13.1 & 0.01 & 4.51 & 0.05 \\ 
 & 3  & 10.8 & 0.09 & 4.22 & 0.06 \\ 
 & 4  & 7.36 & 1.41 & 4.13 & 0.09 \\ 
 & 5  & 6.58 & 20.2 & 4.06 & 0.14 \\ 
 & 6  & 5.52 & 820. & 4.05 & 0.21 \\ 
 & 7  & $-$ & $-$ & 4.04 & 0.25 \\ 
 & 8  & $-$ & $-$ & 4.03 & 0.56 \\ 
\hline
\end{tabular}
\begin{tabular}{|c|c|cc|cc|}
\hline
\multirow{2}*{$n$}  & \multirow{2}*{$d$}   & \multicolumn{2}{c|}{$\rho_d$} &  \multicolumn{2}{c|}{$\rho_d^\#$}   \\
 &    &     value & time & value & time\\
 \hline
\multirow{8}*{5} & 1  & 80.3 & 0.02 & 8.72 & 0.03 \\ 
 & 2  & 32.0 & 0.05 & 6.06 & 0.05 \\ 
 & 3  & 25.4 & 0.42 & 5.46 & 0.07 \\ 
 & 4  & 15.0 & 24.1 & 5.32 & 0.11 \\ 
 & 5  & 12.9 & 1538 & 5.14 & 0.16 \\ 
 & 6  & $-$ & $-$ & 5.10 & 0.35 \\ 
 & 7  & $-$ & $-$ & 5.09 & 0.91 \\ 
 & 8  & $-$ & $-$ & 5.06 & 2.11 \\ 
\hline
\end{tabular}
\\
\begin{tabular}{|c|c|cc|cc|}
\hline
\multirow{2}*{$n$}  & \multirow{2}*{$d$}   & \multicolumn{2}{c|}{$\rho_d$} &  \multicolumn{2}{c|}{$\rho_d^\#$}   \\
 &    &     value & time & value & time\\
 \hline
\multirow{8}*{6} & 1  & 237. & 0.01 & 13.0 & 0.03 \\ 
 & 2  & 80.8 & 0.19 & 7.92 & 0.06 \\ 
 & 3  & 61.8 & 2.61 & 6.90 & 0.10 \\ 
 & 4  & 32.3 & 469. & 6.59 & 0.16 \\ 
 & 5  & $-$ & $-$ & 6.26 & 0.38 \\ 
 & 6  & $-$ & $-$ & 6.19 & 1.18 \\ 
 & 7  & $-$ & $-$ & 6.18 & 4.40 \\ 
 & 8  & $-$ & $-$ & 6.13 & 13.3 \\ 
\hline
\end{tabular}
\begin{tabular}{|c|c|cc|cc|}
\hline
\multirow{2}*{$n$}  & \multirow{2}*{$d$}   & \multicolumn{2}{c|}{$\rho_d$} &  \multicolumn{2}{c|}{$\rho_d^\#$}   \\
 &    &     value & time & value & time\\
 \hline
\multirow{8}*{7} 
 & 1  & 701. & 0.01 & 19.1 & 0.04 \\ 
 & 2  & 209. & 1.24 & 10.2 & 0.06 \\ 
 & 3  & 155. & 21.3 & 8.62 & 0.10 \\ 
 & 4  & $-$ & $-$ & 8.10 & 0.30 \\ 
 & 5  & $-$ & $-$ & 7.43 & 1.17 \\ 
 & 6  & $-$ & $-$ & 7.35 & 4.74 \\ 
 & 7  & $-$ & $-$ & 7.30 & 21.3 \\ 
 & 8  & $-$ & $-$ & 7.26 & 76.0 \\ 
\hline
\end{tabular}
\end{center}
\end{table}

The symbol ``$-$'' indicates that the procedure runs out of memory, which happens during the computation of moments.
The numerical results reported in Table \ref{table:ex1} show that the approach relying on the pushforward measure provides more accurate upper bounds while being much more efficient. 
For this example, the time spent to compute the entries of the two localizing matrices is relatively small (less than 1 \%) compared to the time spent to solve the SDP problem.
Computing $\rho_d$ requires to solve an SDP involving two matrices of size $\binom{n+d}{n}$, while computing $\rho_d^\#$ requires to solve an SDP involving two matrices of size $\binom{2+d}{2}$. 
This explains the efficiency of the method based on the pushforward measure.
Note also that the relative error between $\rho_d^\#$ and $\rho$ increases at fixed $d$ when $n$ increases.
For $n \leq 4$, the obtained results are quite accurate as the relative error  remains below $1 \%$.
For higher values of $n$, the relative error lies between $1.2 \%$ and $3.7 \%$. 

Next, we consider the fraction $\frac{f}{g}$ from Example \ref{ex2}. 
\begin{example}
Let us take
\label{ex2}
\begin{align*}
\hat{f}(\x) = \x^\top \A \x \,, \quad g(\x) = \x^\top \B \x  \,, \quad  \saK= [-1,1]^n \,,
\end{align*}
where $\A$ and $\B$ are matrices with coefficients between $-1$ and $1$, randomly chosen with respect to the uniform distribution and such that $\B$ is positive definite.
Then, we take the minimal evaluation $\hat{\rho}$ of $\frac{\hat f}{g}$ among $10^7$ random points distributed on the cube, and define $f = \hat f - \hat \rho g$ so that  $\frac{f}{g} = \frac{\hat{f}}{g} -  \hat{\rho}$ and $\rho \simeq 0$.
\end{example}

\begin{table}[!ht]
\caption{Upper bounds obtained for the minimum of the function of Example \ref{ex2}.}
\label{table:ex2}
\begin{center}
\begin{tabular}{|p{\sti}|p{\sti}|cc|cc|}
\hline
\multirow{2}*{$n$}  & \multirow{2}*{$d$}   & \multicolumn{2}{c|}{$\rho_d$} &  \multicolumn{2}{c|}{$\rho_d^\#$}   \\
 &    &     value & time & value & time\\
 \hline
\multirow{5}*{2} 
 & 2  & 0.16 & 0.01 & 0.12 & 0.01 \\ 
 & 4  & 0.08 & 0.01 & 0.06 & 0.04 \\ 
 & 6  & 0.05 & 0.04 & 0.02 & 0.18 \\ 
 & 8  & 0.03 & 0.30 & 0.01 & 0.61 \\ 
 & 10  & 0.02 & 1.68 & 0.01 & 2.38 \\ 
\hline
\end{tabular}
\begin{tabular}{|p{\sti}|p{\sti}|cc|cc|}
\hline
\multirow{2}*{$n$}  & \multirow{2}*{$d$}   & \multicolumn{2}{c|}{$\rho_d$} &  \multicolumn{2}{c|}{$\rho_d^\#$}   \\
 &    &     value & time & value & time\\
 \hline
\multirow{5}*{4} 
 & 2  & 0.19 & 0.01 & 0.15 & 0.13 \\ 
 & 3  & 0.14 & 0.09 & 0.12 & 0.19 \\ 
 & 4  & 0.11 & 1.64 & 0.09 & 0.36 \\ 
 & 5  & 0.09 & 26.8 & 0.08 & 0.84 \\ 
 & 6  & 0.07 & 510. & 0.03 & 2.58 \\ 
\hline
\end{tabular}
\\
\begin{tabular}{|p{\sti}|p{\sti}|cc|cc|}
\hline
\multirow{2}*{$n$}  & \multirow{2}*{$d$}   & \multicolumn{2}{c|}{$\rho_d$} &  \multicolumn{2}{c|}{$\rho_d^\#$}   \\
 &    &     value & time & value & time\\
 \hline
\multirow{3}*{6} 
 & 2  & 0.44 & 0.08 & 0.39 & 0.34 \\ 
 & 3  & 0.33 & 2.97 & 0.26 & 1.88 \\ 
 & 4  & 0.26 & 359. & 0.15 & 13.5 \\ 
\hline
\end{tabular}
\begin{tabular}{|p{\sti}|p{\sti}|cc|cc|}
\hline
\multirow{2}*{$n$}  & \multirow{2}*{$d$}   & \multicolumn{2}{c|}{$\rho_d$} &  \multicolumn{2}{c|}{$\rho_d^\#$}   \\
 &    &     value & time & value & time\\
 \hline
\multirow{3}*{8} 
 & 1  & 1.18 & 0.01 & 1.10 & 0.5 \\ 
 & 2  & 0.93 & 0.27 & 0.85 & 1.44 \\ 
 & 3  & 0.74 & 78.3 & 0.62 & 30.3 \\ 
\hline
\end{tabular}
\\
\begin{tabular}{|p{\sti}|p{\sti}|cc|cc|}
\hline
\multirow{2}*{$n$}  & \multirow{2}*{$d$}   & \multicolumn{2}{c|}{$\rho_d$} &  \multicolumn{2}{c|}{$\rho_d^\#$}   \\
 &    &     value & time & value & time\\
 \hline
\multirow{2}*{10} 
 & 1  & 0.99 & 0.39 & 0.92 & 0.47 \\ 
 & 2  & 0.78 & 1.20 & 0.14 & 1.60 \\ 
\hline
\end{tabular}
\begin{tabular}{|p{\sti}|p{\sti}|cc|cc|}
\hline
\multirow{2}*{$n$}  & \multirow{2}*{$d$}   & \multicolumn{2}{c|}{$\rho_d$} &  \multicolumn{2}{c|}{$\rho_d^\#$}   \\
 &    &     value & time & value & time\\
 \hline
\multirow{2}*{12} 
 & 1  & 1.23 & 2.40 & 1.15 & 1.35 \\ 
 & 2  & 1.02 & 4.65 & 0.99 & 22.5 \\ 
\hline
\end{tabular}
\end{center}
\end{table}

The results from Table \ref{table:ex2} show that $\rho_d^\star < \rho_d$, thus the method based on the pushforward measure provides more accurate upper bounds, as previously noticed for Example \ref{ex1}.
When $n$ increases, one can also notice that it is harder to approximate the value of $\rho$.
By contrast with Table \ref{table:ex1}, for $n \geq 10$, the time spent to compute the entries of the localizing matrices becomes larger than the time spent to solve the SDP problem. 
In particular, our implementation lacks of efficiency to compute the support of powers of polynomials $f^i g^j$, when $n$ and the degree $i+j$ of the resulting product gets larger (typically for $n=10$ and  $i+j = 4$).
This explains why the method based on the pushforward measure can be less efficient than the other one.
Implementing an efficient polynomial arithmetic in Julia is left for further development.

\subsection{Sums of rational functions}
Next, we consider a sum of rational functions $\sum_{i=1}^N\frac{f_i}{g_i}$. 
Similarly to the case of a single fraction, we compare the values of $\rho_{d,s}$ and $\rho_{d,s}^\#$ when solving SDP \eqref{eq:rho_sos} and \eqref{eq:rho_pfm_sos}, respectively.
To ease the presentation of the results, we choose $s=d$. 
Despite the potential feasibility issues mentioned in Remark \ref{rk:feas}, the SDP solver always returns a pair of primal-dual optimal solutions.
Table \ref{table:ex3} reports the bounds obtained for the randomly generated functions of Example \ref{ex3} with $N \in \{2,5\}$ and $n \in \{2,5,8\}$.
\begin{example}
\label{ex3}
As in Example \ref{ex2}, we consider for all $i=1,\dots,N$:
\begin{align*}
\hat{f}_i(\x) = \x^\top \A_i \x \,, \quad g_i(\x) = 1+ \x^\top \B_i \x  \,, \quad  \saK= [-1,1]^n  \,,
\end{align*}
where $\A_i$ and $\B_i$ are matrices with coefficients between $-1$ and $1$, randomly chosen with respect to the uniform distribution and such that $\B_i$ is positive definite.
Then, we take the minimal evaluation $\hat{\rho}$ of $ \sum_{i=1}^N\frac{\hat f_i}{g_i}$ among $10^7$ random points distributed on the cube, and consider $f_1 = \hat f_1 - \hat \rho g_1$, $f_2 = \hat f_2, \dots, f_N = \hat f_N$, so that  $\sum_{i=1}^N \frac{f_i}{g_i} = \sum_{i=1}^N \frac{\hat{f_i}}{g_i} -  \hat{\rho}$ and $\rho \simeq 0$.
\end{example}

\begin{table}[!ht]
\caption{Upper bounds obtained for the minimum of the function of Example \ref{ex3}.}
\label{table:ex3}
\begin{center}
\begin{tabular}{|c|c|c|cc|cc|}
\hline
\multirow{2}*{$N$} & \multirow{2}*{$n$}  & \multirow{2}*{$d$}   & \multicolumn{2}{c|}{$\rho_{d,d}$} &  \multicolumn{2}{c|}{$\rho_{d,d}^\#$}   \\
& &    &     value & time & value & time\\
 \hline
\multirow{3}*{2} & \multirow{3}*{2} 
   & 3  & 0.29 & 0.02 & 0.14 & 0.30 \\ 
 & & 4  & 0.21 & 0.04 & 0.10 & 5.41 \\ 
 & & 5  & 0.18 & 0.13 & 0.06 & 115. \\ 
\hline
\end{tabular}
\begin{tabular}{|c|c|c|cc|cc|}
\hline
\multirow{2}*{$N$} & \multirow{2}*{$n$}  & \multirow{2}*{$d$}   & \multicolumn{2}{c|}{$\rho_{d,d}$} &  \multicolumn{2}{c|}{$\rho_{d,d}^\#$}   \\
& &    &     value & time & value & time\\
 \hline
\multirow{3}*{2} & \multirow{3}*{5} 
   & 1  & 1.85 & 0.05 & 1.73 & 0.59 \\ 
 & & 2  & 1.34 & 0.18 & 0.94 & 2.70 \\ 
 & & 3  & 0.95 & 5.61 & 0.32 & 69.4 \\ 
\hline
\end{tabular}
 \\
\begin{tabular}{|c|c|c|cc|cc|}
\hline
\multirow{2}*{$N$} & \multirow{2}*{$n$}  & \multirow{2}*{$d$}   & \multicolumn{2}{c|}{$\rho_{d,d}$} &  \multicolumn{2}{c|}{$\rho_{d,d}^\#$}   \\
& &    &     value & time & value & time\\
 \hline
\multirow{2}*{2} & \multirow{2}*{8} 
   & 1  & 1.42 & 0.53 & 1.29 & 1.58 \\ 
 & & 2  & 1.19 & 6.52 & 1.09 & 225. \\ 
\hline
\end{tabular}
\begin{tabular}{|c|c|c|cc|cc|}
\hline
\multirow{2}*{$N$} & \multirow{2}*{$n$}  & \multirow{2}*{$d$}   & \multicolumn{2}{c|}{$\rho_{d,d}$} &  \multicolumn{2}{c|}{$\rho_{d,d}^\#$}   \\
& &    &     value & time & value & time\\
 \hline
\multirow{2}*{5} & \multirow{2}*{2} 
   & 1  & 1.64  & 0.01 & 1.36 & 0.47 \\ 
 & & 2  & 1.35  & 0.02 & 1.08 & 13.9 \\ 
\hline
\end{tabular}
\\
\begin{tabular}{|c|c|c|cc|cc|}
\hline
\multirow{2}*{$N$} & \multirow{2}*{$n$}  & \multirow{2}*{$d$}   & \multicolumn{2}{c|}{$\rho_{d,d}$} &  \multicolumn{2}{c|}{$\rho_{d,d}^\#$}   \\
& &    &     value & time & value & time\\
 \hline
\multirow{2}*{5} & \multirow{2}*{5} 
   & 1  & 1.08 & 0.15 & 1.03 & 1.50 \\ 
 & & 2  & 1.01 & 0.55 & 0.89 & 239. \\ 
\hline
\end{tabular}
\begin{tabular}{|c|c|c|cc|cc|}
\hline
\multirow{2}*{$N$} & \multirow{2}*{$n$}  & \multirow{2}*{$d$}   & \multicolumn{2}{c|}{$\rho_{d,d}$} &  \multicolumn{2}{c|}{$\rho_{d,d}^\#$}   \\
& &    &     value & time & value & time\\
 \hline
\multirow{2}*{5} & \multirow{2}*{8} 
   & 1  & 1.71 & 0.57 & 1.81 & 17.9 \\ 
 & & 2  & 1.99 & 18.6 & 0.85 & 590. \\ 
\hline
\end{tabular}
\end{center}
\end{table}

Here again, the method based on the pushforward measure suffers from to the lack of efficiency of our implementation to compute the support of powers of polynomials $\prod_{i=1}^N f_i^{\alpha_i} g_i^{\beta_i}$ when the resulting degree gets large.
As for the single fraction case, we still obtain more accurate bounds. 
\section{Conclusion}
\label{sec:conclusion}
\jean{We have described an algorithmic framework for approximating as closely as desired the global minimum of rational fractions over a compact set $\saK$. It consists of a converging hierarchy or approximations indexed by a ``degree" $d$. 
It is based on an equivalent but simpler minimization problem 
in smaller dimension obtained by using the pushforward measure of a reference measure on $\saK$, by some polynomial mapping related to the fractions to be minimized. In case of a single fraction, we obtain a converging hierarchy of upper bounds. For each degree $d$ the 
resulting problem is a generalized eigenvalue problem whose size increases with $d$.
Our numerical preliminary results indicate that this approach provides better approximations
in less computational time.}

\jean{The bottleneck of the method is computing integrals of the form \eqref{eq:integral} to fill up entries of the two matrices involved in the generalized eigenvalue problems that one needs to solve at each step of the hierarchy. Therefore
this approach is currently limited to problems of modest size, with small degree, number of variables and fractions. It is worth mentioning
that if $\saK$ is a simplex then various efficient methods described in \cite{baldoni} can be exploited. One may even expect further progress by 
restricting to certain classes of polynomials and simple sets $\saK$.}

\jean{Another potential numerical issue is the sensitivity of the generalized eigenvalue problem
to solve at each step, with respect to the magnitude of the entries, especially if the matrices are expressed in the usual monomial basis. 
One possible remedy would be to (i)
use a different basis (e.g. basis of tensorized Chebyshev polynomials in the line of research developed in \cite{trefethen2000spectral}) and/or (ii)  rely on recently developed hybrid numeric-symbolic algorithms \cite{lairez2019computing} that yield efficient and certified approximation of integrals.}

\jean{Finally, another interesting issue is to 
provide some convergence rate of the upper bounds hierarchy obtained for the case of single fractions.}

\section*{Acknowledgements}
The work of the first and second authors is supported by the AI Interdisciplinary Institute ANITI funding, through the French ``Investing for the Future PIA3'' program under the Grant agreement n$^{\circ}$ANR-19-PI3A-0004.
The second author was supported by the Tremplin ERC Stg Grant ANR-18-ERC2-0004-01 (T-COPS project), the FMJH Program PGMO (EPICS project) and  EDF, Thales, Orange et Criteo. The research of the third author is conducted in the framework of the regional programme ”Atlanstic 2020, Research, Education and Innovation in Pays de la Loire”, supported
by the French Region Pays de la Loire and the European Regional Development Fund.
This work has benefited from  the European Union's Horizon 2020 research and innovation programme under the Marie Sklodowska-Curie Actions, grant agreement 813211 (POEMA).\\

\appendix

\section{Alternative mappings for a sum of rational functions}
\label{sec:multirat_hierarchy2}

As an alternative to \eqref{eq:mapping1}, consider the mapping 
\begin{equation}
\begin{split}
\mathbf{U}: \saK &\rightarrow \mathbb{R}^N\\
\mathbf{x} &\mapsto \begin{bmatrix}
\frac{f_1(\mathbf{x})}{g_1(\mathbf{x})}\\
\cdots \\
\frac{f_N(\mathbf{x})}{g_N(\mathbf{x})} 
\end{bmatrix}\in\mathbb{R}^N,
\end{split}
\end{equation}
and the uniform measure $\lambda \in \mathcal{M}(\saK)$.
As for the case of a single fraction, one sets $\mathbf{z}:=(z_1,\dots,z_N) = 
\left(\frac{f_1(\mathbf{x})}{g_1(\mathbf{x})}, \dots,  \frac{f_N(\mathbf{x})}{g_N(\mathbf{x})} \right)$. 
For all $\alpha \in \mathbb{N}^N$:
\begin{equation}
\label{eq:mom-gamma}
\gamma_{\alpha} := \int_{\mathbb{R}^N} \mathbf{z}^{\alpha} \d \U_\# \lambda (\z) = \int_{\saK} 
\left(\frac{f_1(\mathbf{x})}{g_1(\mathbf{x})}\right)^{\alpha_1}  \cdots  \left(\frac{f_N(\mathbf{x})}{g_N(\mathbf{x})}\right)^{\alpha_N} \d\lambda (\x).  
\end{equation}

By definition of the mapping $\mathbf{U}$, we can rewrite \eqref{quotients} as follows
\begin{equation}
\label{optU}
\rho := \min_{\z \in \U(\saK)} z_1 + \dots + z_N.
\end{equation}
Given any $d\in\mathbb{N}$, let $\gammab := (\gamma_{\alpha})_{\alpha \in \mathbb{N}_{2d}^N}$ and let us consider the following problem
\begin{align}
\label{gen-eig-multi}
\begin{array}{rl}
a_d^\# = \displaystyle\sup_{a \in \R}& \,a: \\
\mbox{s.t.}& \M_{d}((z_1 + \dots + z_N) \gammab) \, \succeq \, a \, \M_d(\gammab) \,.
\end{array}
\end{align}
The next result can be proved as in  Theorem~\ref{th:cvg_rat_pfm}:
\begin{theorem}
\label{thm-convergence-multi}
Consider the hierarchy of semidefinite programs \eqref{gen-eig-multi}, indexed by $d\in\mathbb{N}$. Then:
\begin{itemize}
\item[(i)] SDP \eqref{gen-eig-multi} has an optimal solution $a_d^\# \geq \rho$ for every $d\in\mathbb{N}$.
\item[(ii)] The sequence $(a_d^\#)_{d\in\mathbb{N}}$ is monotone nonincreasing and $a_d^\# \downarrow \rho$ as $d\rightarrow +\infty$. 
\end{itemize}
\end{theorem}
In the sequel, we present three different possible frameworks to approximate the entries of $\gammab$ in \eqref{eq:mom-gamma}. 
\subsection{First framework via an SDP hierarchy}
\label{integral_sphere}
For a fixed $\alpha \in \N^n$, we present a first iterative scheme to approximate the moments
$\gammab=(\gamma_\alpha)$ in \eqref{eq:mom-gamma}.
%
Note that computing $\gamma_{\alpha}$ boils down to solving a particular instance of the generic problem
\[I \,:=\,\int_{\saK}\frac{f(\x)}{g(\x)}\,\d\lambda(\x),\]
%
%
Introduce the measure $\mu$ such that $\int_{\saK} \x^\alpha g(\x)\,\d\mu(\x)= \int_{\saK} \x^\alpha\,\d\lambda(\x)$ for all $\alpha\in\N^n$ so that
\[I \,=\,\int_{\saK} f(\x)\,\d\mu(\x).\]
The moments of $\mu$ can be approximated from the moments $(\lambda_\alpha)$ of $\lambda$ by following the approach in 
\cite[\S 12.1.1]{lasserre2010moments}.
Let $\theta(\x)=1-\Vert\x\Vert^2$, and with $r\in\N$ fixed, solve:
\[a_r:=\inf_\y \,\{\,{\rm trace}({\M}_r(\y))\,:\: \M_r(\y)\succeq0,\,\M_{r-1}(\theta\,\y)=0\,;\: \sum_{\beta} g_\beta \,y_{\beta+\alpha}=\lambda_\alpha,\quad\vert\alpha\vert\leq 2r-d_g\},\]
and let $\y^{r}$  be an optimal solution. Then with $d\in\N$ fixed, arbitrary:
\[\lim_{r\to\infty}\,\sup_{\vert\alpha\vert\leq d} \vert y^{r}_\alpha-\mu_\alpha\,\vert \,=\,0 \,.\]
In doing so, one can obtain arbitrary close approximations of any fixed number of moments of $\mu$, which in turn provides a converging scheme to compute $I$ and in particular $\gamma_{\alpha}$.

\subsection{Second framework via the generating function}
\label{vandermonde}
\jean{For fixed $d \in \N$, we present a second iterative scheme to compute $\gammab= (\gamma_{\alpha} )_{|\alpha|\leq d}$ in \eqref{eq:mom-gamma}.
Let $h_j := \frac{f_j}{g_j}$, for each $j=1,\dots,N$. Define
the function $\phi : [0,1]^{N} \to \C$ by:}
\[ \phi(t_1,\dots,t_N) := \int_{\saK} \exp \left( \sum_{j=1}^N (\exp( 2 \pi i t_j ) h_j(\x)) \right) d\x \,, \]
and let $\t := (t_1,\dots,t_N)$. 
After performing Taylor expansion of the exponential function at order $r \geq d$, one obtains
\[
\phi(\t) = \underbrace{\sum_{\mid \alpha \mid \leq r} \exp( 2 \pi i \t \cdot \alpha ) \binom{\mid \alpha \mid }{\alpha} \gamma_{\alpha}}_                        {\phi_r(\t)} + e_r(\t) \,,
\]
where $e_r$ denotes the corresponding Taylor remainder.
With $s_{r,N} := \binom{N+r}{r}$, we can evaluate either $\phi$ or $\phi_r$ randomly at $s_{r,N}$ points of $[0,1]^N$. 
Denoting by $\phib$ and $\phib_r$ the respective vectors of values, we obtain
\begin{equation}
\label{eq:vandermonde}
\phib_r = \V_r \gammab \,,
\end{equation}
where $\V_r$ is the multivariate matrix with entries $\left(\exp( 2 \pi i \t \cdot \alpha ) \binom{\mid \alpha \mid }{\alpha}\right)_{\mid \alpha \mid \leq r}$.
We obtain an approximation $\gammab^r$ of  $\gammab$ by solving the following linear system of equations, instead of \eqref{eq:vandermonde}:
\begin{equation}
\label{eq:approx_vandermonde}
\phib = \V_r  \gammab^r \,.
\end{equation}
Then with $d\in\N$ fixed, arbitrary, one can show that
\begin{align}
\lim_{r\to\infty}\,\sup_{\vert\alpha\vert\leq d} \vert \gamma_\alpha^r-  \gamma_\alpha\,\vert \,=\,0 \,.
\end{align}


\subsection{Third framework via a link with the Gaussian}
Here we assume that $\saK=\mathbb{S}^{n-1}$ and $f,g$ are positively  homogeneous functions of degree $d_f$ and $d_g$ respectively.
Let $\E_n:=\{\x:\:\Vert\x\Vert\leq 1\}$ (the Euclidean unit ball). 
Then
\begin{eqnarray*}
\int_{\mathbb{S}^{n-1}}\x^\alpha\,\frac{f}{g}\,d\lambda&=&
(n+\vert\alpha\vert +d_f-d_g)\,\int_{\E_n}\,\x^\alpha\,\frac{f}{g}\,d\x\\
&=&\frac{n+\vert\alpha\vert+d_f-d_g}{\Gamma(1+(n+\vert\alpha\vert+d_f-d_g)/2)}\int_{\R^n}\,
\x^\alpha\,\frac{f}{g}\,\exp(-\Vert\x\Vert^2)d\x\end{eqnarray*}
for all $\alpha\in\N^n$.
Therefore one may use any method to approximate integrals with respect to Gaussian measure to obtain integrals of products of rational functions.
In particular observe that if $d_f=d_g$ then
\[\int_{\bS^{n-1}}\left(\frac{f}{g}\right)^j\,d\lambda\,=\,
\frac{n}{\Gamma(1+n/2)}\int_{\R^n}\,\left(\frac{f}{g}\right)^j\,\exp(-\Vert\x\Vert^2)\,d\x.\]
Similarly:
\[\int_{\bS^{n-1}}\left(\frac{f_1}{g_1}\right)^i\left(\frac{f_2}{g_2}\right)^j
\,d\lambda\,=\,\frac{n}{\Gamma(1+n/2)}\int_{\R^n}\,\left(\frac{f_1}{g_1}\right)^i\left(\frac{f_2}{g_2}\right)^j
\exp(-\Vert\x\Vert^2)\,d\x.\]

\if{
\if{
Consider the mapping $\U:\bS^{n-1}\to\R^{2N}$ defined by:
\[\x\mapsto \U(\x)=(f_1(\x),g_1(\x),\ldots f_N(\x),g_N(\x)),\]
and let $d\lambda^\# (u_1,v_1,\ldots,u_{N},v_N)$ (on $\R^{2N}$) be the pushforward of $\lambda$ on $\bS^{n-1}$ by $\U$.
Then
\[=\min_{\x\in\bS^{n-1}} \sum_{i=1}^N\frac{f_i(\x)}{g_i(\x)}\,=\,\min_{(\u,\v)\,\in{\rm supp}(\lambda^\#)}\,\sum_{i=1}^n\frac{u_i}{v_i}.\]
Then for every $d\in\N$ define:
\[\begin{array}{rl}
\rho_d=\displaystyle\min_{\sigma,\phi_1,\ldots,\phi_N}& \left\{\,\displaystyle\sum_{i=1}^N \int u_i\,\sigma\,d\phi_i: \right.\\
\mbox{s.t.}&\displaystyle\int \u^\alpha\v^\beta\,v_i\,d\phi_i\,=\,\displaystyle\int \u^\alpha\v^\beta\,d\lambda^\#,\quad 
\vert\alpha+\beta\vert\,\leq 2d;\:i\leq N\\
&\displaystyle\int\sigma\,d\lambda^\#\,=\,1\\
&\left.\sigma\in\Sigma[\u,\v]_d\,\right\}.
\end{array}\]
\begin{lemma}
Assume that $g_i>0$ (or $g_i<0$) for all $\x\in\mathbb{S}^{n-1}$. Then
$\rho_d\downarrow \rho$ as $d\to\infty$.
\end{lemma}

It suffices to observe that
\[\rho=\displaystyle\inf_{\sigma\in\Sigma[\u,\v]}\,\{\int \sum_{i=1}^N\frac{u_i}{v_i}\sigma\,d\lambda^\#:\:\int\sigma\,d\lambda^\#=1\},\]
and if 
\[\int\u^\alpha\,\v^\beta\,v_i\,d\phi_i\,=\,\int\u^\alpha\,\v^\beta\,d\lambda^\#,\quad \forall (\alpha,\beta),\quad i=1,\ldots,N,\]
then $v_i\,d\phi_i=d\lambda^\#$ for all $i=1,\ldots N$, and therefore:
\begin{eqnarray*}
\displaystyle\sum_{i=1}^n\,\displaystyle\int u_i\,\sigma\,d\phi_i&=&
\displaystyle\sum_{i=1}^n\int\frac{u_i}{v_i}\,\sigma\,v_i\,d\phi_i\\
&=&\displaystyle\sum_{i=1}^n\displaystyle\int \frac{u_i}{v_i}\,\sigma\,d\lambda^\#
\end{eqnarray*}
Of course this works if $N$ is relatively small .... and one knows how to compute (or approximate well)
\[\displaystyle\int_{\mathbb{S}^{n-1}}\prod_{i=1}^N f_i(\x)^{\alpha_i}g_i(\x)^{\beta_i}\,d\mu(\x)\,=:\,\int \u^\alpha\,\v^\beta\,d\lambda^\#\]
where $\mu$ is the rotation invariant measure on $\mathbb{S}^{n-1}$ (at least for 
all $\vert\alpha+\beta\vert\leq 2d$).
}\fi
}\fi

\end{document}